\documentclass[reqno,final]{amsart-mod}
\usepackage{natbib}  
\usepackage{fancyhdr} 
\usepackage{color} 
\usepackage{hyperref} 
\usepackage{graphicx} 

\allowdisplaybreaks[4]
\usepackage{bbm}
\usepackage{mathabx}
\usepackage{mathtools}
\usepackage{breakurl}
\usepackage{tikz}

\definecolor{aleacolor}{rgb}{0.16,0.59,0.78}

\hypersetup{
breaklinks,
colorlinks=true,
linkcolor=aleacolor,
urlcolor=aleacolor,
citecolor=aleacolor}

\renewcommand{\headheight}{11pt}
\renewcommand{\cite}{\citet}

\theoremstyle{plain}
\newtheorem{theorem}{Theorem}[section]                                          
                          
\newtheorem{lemma}[theorem]{Lemma}

\theoremstyle{definition}
\newtheorem{definition}[theorem]{Definition}
\theoremstyle{remark}
\newtheorem{remark}[theorem]{Remark}

\makeatletter \@addtoreset{equation}{section} \makeatother
\renewcommand\theequation{\thesection.\arabic{equation}}
\renewcommand\thefigure{\thesection.\arabic{figure}}
\renewcommand\thetable{\thesection.\arabic{table}}

\renewcommand{\Pr}{\mathbb{P}}
\newcommand{\N}{\mathbb{N}}
\newcommand{\Z}{\mathbb{Z}}
\newcommand{\R}{\mathbb{R}}
\newcommand{\E}{\mathbb{E}}

\newcommand{\CB}{\mathcal{B}}
\newcommand{\CC}{\mathcal{C}}
\newcommand{\CD}{\mathcal{D}}

\newcommand{\CH}{\mathcal{H}}
\newcommand{\CL}{\mathcal{L}}
\newcommand{\CM}{\mathcal{M}}

\newcommand{\CP}{\mathcal{P}}
\newcommand{\CW}{\mathcal{W}}
\newcommand{\CX}{\mathcal{X}}

\newcommand{\CZ}{\mathcal{Z}}
\newcommand{\e}{\varepsilon}

\DeclareMathOperator{\Var}{Var}

\definecolor{cff0000}{RGB}{255,0,0}

\newcommand{\norm}[1]{\lVert#1\rVert} 
\newcommand{\Norm}[1]{\big\lVert#1\big\rVert}
\newcommand{\ind}[1]{\mathbbm{1}_{\{#1\}}} 
\newcommand{\indset}[1]{\mathbbm{1}_{#1}} 

\newcommand{\sgn}{\mathop{\mathrm{sgn}}}

\renewcommand{\l}{\left}
\renewcommand{\r}{\right}
\newcommand{\ra}{\rightarrow}

\newcommand{\lra}{\longrightarrow}

\title[Brownian web limit for coalescing random walks on an oriented percolation
cluster]{Coalescing directed random walks on the backbone of a
  $1+1$-dimensional oriented percolation cluster converge to the
  Brownian web} \date{\today}
\author{Matthias Birkner}
\author{Nina Gantert}
\author{Sebastian Steiber}

\address{Johannes Gutenberg University Mainz,
  Institute for Mathematics\newline
  Staudingerweg 9, 55099 Mainz, Germany
}
\email{birkner@mathematik.uni-mainz.de}

\address{Technical University Munich, Fakult\"{a}t f\"{u}r Mathematik\newline
  Boltzmannstra\ss{}e 3, 85748 Garching, Germany
}
\email{gantert@ma.tum.de}

\address{Johannes Gutenberg University Mainz,
  Institute for Mathematics\newline
  Staudingerweg 9, 55099 Mainz, Germany
}
\email{sebastian.steiber@gmail.com}


\begin{document}
\maketitle

\begin{abstract}
  We consider the backbone of the infinite cluster generated by
  supercritical oriented site percolation in dimension $1+1$. A
  directed random walk on this backbone can be seen as an ``ancestral
  lineage'' of an individual sampled in the stationary discrete-time
  contact process. Such ancestral lineages were investigated in
  \cite{BCDG13} where a central limit theorem for a single walker was
  proved.  Here, we consider infinitely many coalescing walkers on the
  same backbone starting at each space-time point. We show that, after
  diffusive rescaling, the collection of paths converges in
  distribution (under the averaged law) to the Brownian web.  Hence,
  we prove convergence to the Brownian web for a particular system of
  coalescing random walks in a dynamical random environment.  An
  important tool in the proof is a tail bound on the meeting time of
  two walkers on the backbone, started at the same time.  Our result
  can be interpreted as an averaging statement about the percolation
  cluster: apart from a change of variance, it behaves as the full
  lattice, i.e. the effect of the ``holes'' in the cluster vanishes on
  a large scale.
\end{abstract}

\section{Introduction}

Informally, the Brownian web is a system of one-dimensional coalescing
Brownian motions starting from every point in space and time. It was
first introduced in \cite{Arratia:79}, studied rigorously in
\cite{TothWerner:1998}, \cite{FINR:02} and has since then been shown
to be a scaling limit of many 1+1-dimensional coalescing structures.
See also \cite{SchertzerSunSwart2015} for an overview, historical
discussion and references.  Possibly the most natural example that
comes to mind in this respect is the system of coalescing random walks
on $\Z$ which is dual to the one-dimensional voter model (see, e.g.,
\cite{Liggett:1999}). This was shown to converge to the Brownian web
(in \cite{FINR:02} for the nearest neighbor case and in
\cite{NewmanRavishankarSun2005} in the general case).
One often interprets the voter model as a population model in which there is
always exactly one individual at each site $x \in \Z$, which can be of one of
two possible types, say. The dual system of random walks
is then naturally interpreted as ancestral lines of the individuals.
Note that while the total population is infinite,
the local population size at a site in the voter model is fixed (at one). 

There is interest in spatial population models with randomly
fluctuating local population sizes, see, e.g., \cite{Etheridge:2004},
\cite{FournierMeleard:2004}, \cite{Etheridge:2006}, \cite{BCD2016} and
the discussion and references there. In this case, ancestral lines 
are random walks in a dynamic random environment which is given by the
time reversal of the population model. 
\cite{BCDG13} considered the specific but prototypic example of the
stationary supercritical discrete time contact process. 
Its time-reversal is the backbone of the supercritical oriented
percolation cluster and in \cite{BCDG13}, a central limit theorem was
proved for such a walk, i.e., for a single ancestral lineage.
It is then a natural problem to study the joint behavior of several
or in fact of all ancestral lineages, hence a system of coalescing random
walks in a dynamic random environment. We address this problem here in the case
$d=1$. Our main result, Theorem~\ref{main2} below, shows then that on
large scales, the effect of the local population fluctuations manifests itself
only as a scaling factor compared to the case of fixed local sizes.
This in a sense rigorously confirms the approach that is often taken
in modelling spatially distributed biological populations where one exogenously fixes the local
population size by considering so-called stepping stone models,
see, e.g., \cite{Kimura:1953}, \cite{Wilkinson-Herbots:1998}.
See also Section~\ref{sect:outlook} below for
more details on the relation to the discrete time contact process
and also \cite{BCD2016} for discussion and a broader class of examples.

\subsection{Set-up} 
Let $\omega\coloneqq \{\omega(x,n): (x,n)\in \Z \times \Z\}$ be
i.i.d.\ $\mathrm{Bernoulli}(p)$ random variables.  A space-time site
$(x,n) \in \Z \times \Z$ is said to be \emph{open} if $\omega(x,n)=1$
and \emph{closed} if $\omega(x,n)=0$.  A directed \emph{open path}
from $(x,m)$ to $(y,n)$ for $m \le n$ is a sequence $x_m,\dots, x_n$
such that $x_m=x$, $x_n=y$, $|x_k-x_{k-1}| \le 1$ for
$k=m+1, \dots, n$ and $\omega(x_k,k)=1$ for all
$k=m,\dots,n$. 
We write $(x,m) \mathop{\to}\limits^\omega (y,n)$ if such an open path
exists and $(x,m) \mathop{\to}\limits^\omega \infty$ if there exists
at least one infinite directed open path starting at $(x,m)$.
\smallskip

There is $p_c = p_c(1) \in (0,1)$ such that
$\Pr\big((0,0) \mathop{\to}\limits^\omega \infty \big) > 0$ if and
only if $p > p_c$ (see e.g.\ Theorem~1 in
\cite{GrimmettHiemer:02}). We assume from now on that $p>p_c$.  Let
\begin{align}
  \label{def:cluster}
  \CC &\coloneqq \big\{(x,n) \in \Z \times \Z : (x,n) \mathop{\to}\limits^\omega \infty  \big\}
        \notag \\
  & \, = \big\{ (x,n) \in \Z \times \Z \, : \, \forall \, k > n \: \exists \, y \in \Z \text{ such that } 
  (x,n) \mathop{\to}\limits^\omega (y,k) \big\}
\end{align}
be the backbone of the space-time cluster of oriented percolation (note that $\CC$ is a function of $\omega$ 
and $|\CC|=\infty$ a.s.\ for $p>p_c$). 
\medskip

We consider walks
$X^{(x_0,t_0)}=\big(X^{(x_0,t_0)}_t \big)_{t \in \Z, t \geq t_0}$
starting 
at any space-time point $(x_0,t_0) \in \Z \times \Z$ and moving as
directed simple random walk on $\CC$. More precisely, let
\begin{equation}
  \label{Udef}
  U(x) \coloneqq \{ y \in \Z : | x-y | \le 1\}
\end{equation}
be the $\ell_\infty$-neighbourhood of site $x \in \Z$ and let $\widetilde{\omega} = \big(
\widetilde\omega{(x,n)} : x \in \Z^d, n \in \Z \big)$, where
$\widetilde{\omega}(x,n) = \big( \widetilde{\omega}(x,n)[1],
\widetilde{\omega}(x,n)[2], \widetilde{\omega}(x,n)[3]\big)$
is a uniformly chosen permutation of $U(x)$, independently
distributed for different $(x,n)$'s and independent of the
$\omega$'s. Define
\begin{align} 
\label{eq:Phidef}
\Phi(x,n) \coloneqq 
\begin{cases} 
  \widetilde{\omega}(x,n)\big[ \min\{ i : (\widetilde{\omega}(x,n)[i],n+1) \in \CC \}\big], & 
  \text{if } \CC \cap \big( U(x) \! \times \! \{n+1\}\big) \neq \emptyset, \\
  \widetilde{\omega}(x,n)[1], & \text{otherwise}.
\end{cases}
\end{align} 
Note that when $(x,n) \in \CC$, the first case occurs and $\Phi(x,n)$
is a uniform pick among those sites in $\{ y : (y,n+1) \in \CC \}$,
the $n+1$-time slice of $\CC$; when $(x,n)
\not\in \CC$, $\Phi(x,n)$ is simply a uniformly chosen neighbour of
$x$. We put 
\begin{align} 
  \label{eq:defXdynamics}
  X^{(x_0,t_0)}_{t_0} \coloneqq x_0, \quad X^{(x_0,t_0)}_{t+1} \coloneqq \Phi(X^{(x_0,t_0)}_t, t), \; t \in \Z_+. 
\end{align}
For fixed $(x_0,t_0) \in \Z \times \Z$, given $\omega$, $X^{(x_0,t_0)}$ is a (time-inhomogeneous) 
Markov chain with 
\begin{align}
  \label{eq:defXtransprob}
  & P_{\omega}( X^{(x_0,t_0)}_{t+1}=y \, \vert \, X^{(x_0,t_0)}_{t}=x ) \notag \\
  & = \indset{U(x)}(y) \times \begin{cases}
    {|(U(x) \times \{ t+1\}) \cap \CC|}^{-1}  & \text{if } (y,t+1) \in \CC,\\[0.7ex]
    0 & \parbox{14em}{if $(y,t+1) \not\in \CC$ but\\ $(U(x) \times \{ t+1\}) \cap \CC \neq \emptyset$,} \\[1.8ex]
    |U(x)|^{-1} & \text{if }  (U(x) \times \{ t+1\}) \cap \CC = \emptyset
  \end{cases}
\end{align}
and $P_{\omega}(X^{(x_0,t_0)}_{t_0}=x_0)=1$. 
In fact, \eqref{eq:defXdynamics} implements a (coalescing) stochastic flow 
with individual paths having transition probabilities given by \eqref{eq:defXtransprob}.

When $t_0=0$ is fixed, we will abbreviate $X^{(z)} \coloneqq X^{(z,0)}$ for $z \in \Z$. 
\medskip

This walk was introduced and studied in \cite{BCDG13}, we refer to
that paper for a more thorough discussion of the background and
related works.  In particular, \cite{BCDG13} describe a regeneration
construction for $X^{(x_0,t_0)}$ and derived a LLN and a quenched CLT
from it, see Theorems~1.1 and 1.3 there; the results also imply that $X^{(x_0,t_0)}$ and
$X^{(x_1,t_0)}$ are ``almost independent'' when they are far apart. We
recall in Section~\ref{sect:preliminaries} below some details from
\cite{BCDG13} that are relevant for the present study, see in particular \eqref{eq:sigmasq.formula} 
for the non-trivial variance in the CLT.  Thus, we
expect that on sufficiently large space-time scales, any 
collection $X^{(x_0,t_0)}, X^{(x_1,t_1)}, \dots, X^{(x_n,t_n)}$ should
look similar to (coalescing) random walks.

\begin{remark}
The study of random walks in dynamic random environments
is currently a very active field which we cannot survey completely here, see e.g.\
\cite{Avena+al:2011}, \cite{Hilario+al:2015}, \cite{BethuelsenVoellering:2016},
\cite{SalviSimenhaus:2018} and the references there for recent examples.
We note however that the walks we consider here are somewhat unusual with
respect to that literature because of the time directions: There, one often
considers scenarios where both the walk and the random environment have the same
``natural'' forwards in time direction as a (Markov) process whereas in our case, forwards in time
for the walk means backwards in time for the environment, namely the discrete time contact process.
More precisely, the “time-slices” of the cluster $\CC$ can be seen to be equal in distribution to the time-reversal of a
stationary discrete-time contact process $(\eta_n)$, we refer to \cite{BCDG13} for details. 
\end{remark}

\subsection{Main result: Brownian web limit in \texorpdfstring{$d=1$}{d=1}}

Before stating our main result we briefly recall a suitable definition
of the Brownian web, following for example \cite{FINR:02} or
\cite{Sun:05}.  See also \cite{SchertzerSunSwart2015} for a broader
introduction and an overview of related work.  We define a metric on
$\R^2$ by
\begin{equation*}
  \rho((x_1,t_1),(x_2,t_2)):=|\tanh(t_1)-\tanh(t_2)|\vee\l|\frac{\tanh(x_1)}{1+|t_1|}-\frac{\tanh(x_2)}{1+|t_2|}\r|
\end{equation*}
Let $R^2_c$ be the completion of $\R^2$ under $\rho$. We can think of
$R^2_c$ as the image of $[-\infty,\infty]\times[-\infty,\infty]$ under
the mapping
\[(x,t)\mapsto 
\l(\frac{\tanh(x)}{1+|t|},\tanh(t)\r)\in R^2_c,\]
i.e., $R^2_c$ can be identified with the square
$[-1,1]\times[-1,1]$ where the line $[-1,1]\times\{1\}$ and the line
$[-1,1]\times\{-1\}$ are squeezed to two single points which we call
$(*,\infty)$ and $(*,-\infty)$.

We define $\Pi$ to be the set of functions
$f:[\sigma,\infty]\lra[-\infty,\infty]$ with ``starting points''
$\sigma\in[-\infty,\infty]$, such that the mapping
$t\mapsto(f(\sigma\vee t),t)$ from $(\R,|\cdot|)$ to $(R^2_c,\rho)$ is
continuous. We consider the elements in $\Pi$ as a tuple of the
function $f$ and its starting point $\sigma$.  The set $\Pi$ together
with the metric
\begin{align} 
  \label{def:dmetric}
  d((f,\sigma),(g,\sigma')) 
  \coloneq 
  |\tanh(\sigma)-\tanh(\sigma')|\vee \hspace{-0.4em} \sup_{t \ge \sigma \wedge \sigma'} \l|\frac{\tanh(f(t\vee\sigma))}{1+|t|}-\frac{\tanh(g(t\vee\sigma'))}{1+|t|}\r|
\end{align}
becomes a complete separable metric space.  Let $\CH$ be the set of
compact subsets of $(\Pi,d)$. 
Equipped with the Hausdorff metric
\[d_{\CH}(K_1,K_2):=\sup_{(f,\sigma)\in K_1}\inf_{(g,\sigma')\in
  K_2}d((f,\sigma),(g,\sigma'))\vee \sup_{(g,\sigma')\in
  K_2}\inf_{(f,\sigma)\in K_1}d((f,\sigma),(g,\sigma')),
\]
$\CH$ is a complete separable metric space.  Let $\CB_\CH$ be the
Borel $\sigma$-algebra associated with the metric $d_{\CH}$. We can
characterize the Brownian web (BW) as follows:

\begin{definition}[Brownian web]
  The Brownian web is a $(\CH,\CB_\CH)$-valued random variable $\CW$,
  whose distribution is uniquely determined by the following properties:
  \begin{itemize}
  \item[(i)] For each deterministic $z\in \R^2$, the set
    $\CW(z):=\{(f,\sigma)\in \CW: (f(\sigma),\sigma)=z\}$ contains
    exactly one element almost surely.
  \item[(ii)] For all $z_1,...,z_k\in \R^2$, $(\CW(z_1),...,\CW(z_k))$
    is distributed as coalescing Brownian motions.
  \item[(iii)] For any countable and dense subset $D$ of $\R^2$,
    almost surely, $\CW$ is the closure of $\{\CW(z):z\in D\}$ in
    $(\Pi,d)$.
  \end{itemize}
\end{definition}

Let us give a precise definition of the system of coalescing random
walks starting from each point contained in the space-time-cluster of
oriented percolation: Let
$\CC=\{(x,n)\in \Z\times\Z:(x,n)\mathop{\to}\limits^\omega \infty\}$
be the set of all points in the space-time lattice which are connected
to infinity (as defined in \eqref{def:cluster}). If a space-time point
$z=(x,n)\in\Z\times\Z$ is in $\CC$ let
\begin{align}
  \begin{minipage}{0.8\textwidth}
    $\pi^z$ be the linearly
    interpolated path of the random walk $X^{(z)}$\\ starting from $z$ with
    dynamics \eqref{eq:defXdynamics}.
  \end{minipage} 
\end{align}
If a point $z\in\Z\times\Z$ is not in $\CC$, we choose the next point
to the left of $z$ that is connected to infinity and define $\pi^z$ as
a linearly interpolated copy of the path starting there. In formulas,
if $z=(x,n)\notin \CC$ we define
\begin{equation}
  \label{next point left}
  c((x,n)):=\max\{y \leq x : (y,n)\in \CC\}\text{ and }(\pi^z(t))_{t\geq n}:=(\pi^{c(z)}(t))_{t\geq n}.
\end{equation}

Let $\mathbf{\Gamma}$ be the collection of all paths, i.e.\ 
\begin{equation}
  \label{def:Gamma}
  \mathbf{\Gamma}:=\{\pi^z:z\in \Z\times\Z\}=\{\pi^z:z\in \CC\}.
\end{equation}
Since all paths in $\mathbf{\Gamma}$ are equicontinuous the closure of
$\mathbf{\Gamma}$, which we also denote by $\mathbf{\Gamma}$, is a
random variable taking values in $(\CH,\CB_\CH)$. 
\smallskip

In order to formulate the convergence theorem precisely we 
consider for $\delta > 0$ and $b>0$ ($b$ normalizes the 
standard deviation) the \emph{diffusive scaling map}
\begin{equation*}
  S_{b,\delta}:=(S^1_{b,\delta},S^2_{b,\delta}):(R^2_c,d)\lra(R^2_c,d),
\end{equation*}
where
\begin{align*}
  S_{b,\delta}(x,t):=(S^1_{b,\delta}(x,t),S^2_{b,\delta}(t)):=
  \begin{cases}
    (\frac{x\delta}{b},\delta^2 t),&\text{ if }(x,t)\in\R^2,\\
    (\pm\infty,\delta^2 t),&\text{ if }(x,t)=(\pm\infty,t), t\in\R,\\
    (*,\pm\infty),&\text{ if }(x,t)=(*,\pm\infty).
  \end{cases}
\end{align*}
The mapping $S_{b,\delta}$ is naturally extended to $(\Pi,d)$ via
\begin{align*}
  S_{b,\delta}:\Pi&\lra\Pi\\
  (\pi,t)&\mapsto(S^1_{b,\delta}\circ\pi,S^2_{b,\delta}\circ t).
\end{align*}
For $K \subset \Pi$ we set
$S_{b,\delta}K:=\{S_{b,\delta}((\pi,t)):(\pi,t)\in K\}$. Note that
$K\in\CH$ implies $S_{b,\delta}K\in\CH$. 
\medskip

\setcounter{theorem}{-1}
\begin{theorem}[\cite{BCDG13}]
  \label{thm0}
  There is $v \in (0,\infty)$ such that 
  conditioned on $(0,0)\in\CC$, 
  \[
  S_{v,\delta}\pi^{(0,0)} \mathop{\longrightarrow}^d_{\delta \downarrow 0} \; \text{standard Brownian motion}.
  \]
\end{theorem}
The variance $v$ has a description in terms of regeneration times,
which we recall from \cite{BCDG13} in \eqref{eq:sigmasq.formula} below
(cf.\ \cite[\text{Remark }1.2]{BCDG13}). 
\smallskip

Our main result is the following theorem.
\begin{theorem}
  \label{main2}
  The $(\CH,\CB_\CH)$-valued random variables
  $(S_{v,\delta}\mathbf{\Gamma})$ 
  converge in distribution to the
  Brownian web as $\delta \downarrow 0$.
\end{theorem}
\begin{remark}
  \label{rem:aftermain}
  1.\ An analogous result holds when $U(x) = \{ y : |y-x| \leq R\}$
  for some $R \in \N$.  Furthermore, note that even for $R=1$, paths
  in $\mathbf{\Gamma}$ can cross each other without coalescing.
  \smallskip
  
  \noindent 2.\ In the parlance of random walks in random
  environments, Theorems~\ref{thm0} and \ref{main2} are
  \emph{annealed} limit theorems, i.e., the randomness refers to
  jointly averaging the walk and the realization of the percolation
  cluster.  In fact, \cite{BCDG13} proved also a \emph{quenched}
  version of Theorem~\ref{thm0}, where a typical cluster is fixed and
  randomness refers only to the steps of the walk.  However, we
  presently do not have a quenched analogue of Theorem~\ref{main2}
  (see also the discussion in Section~\ref{sect:outlook} below).
  \smallskip
  
  \noindent 3.\ \cite{SarkarSun:2013} considered the system of
  \emph{rightmost} paths on an oriented (bond) percolation cluster and
  showed that it converges to the Brownian web after suitable
  centering and rescaling. Thus, in \cite{SarkarSun:2013}, walkers
  move to the right whenever possible (and in particular they cannot
  cross each other) whereas in our set-up, the walks pick uniformly
  among the allowed neighbors.
\end{remark}
We prove Theorem~\ref{main2} in Section~\ref{sect:proof} and discuss
some implications and further questions in Section~\ref{sect:outlook}.

\section{Proofs}
\label{sect:proof}
\begin{remark}
  In the proofs that follow $C$ and $c$ denote some positive constants whose exact value is not important for the argument. 
  The constants $C$ and $c$ may also vary within a chain of inequalities. If the value of a certain constant is important for a later step, we add a subscript to it $C_1,C_2,...$.
\end{remark}

\subsection{Preliminaries} 
\label{sect:preliminaries}
Here, we briefly recall concepts and results from \cite{BCDG13} that will be required for our arguments.

For $z=(x,n) \in \Z \times \Z$, writing $B_z \coloneq \{ z \in \CC
\}$, we abbreviate 
\begin{align}
  \widetilde\Pr_z(\cdot) := \Pr(\cdot \,|B_z) \quad \text{and} \quad 
  \widetilde\Pr_{z_1,z_2}(\cdot) := \Pr(\cdot \,|B_{z_1}\cap B_{z_2}) .
\end{align}

\cite[Sections~2.1--2.2]{BCDG13} describes a regeneration
construction for $(X^z_t)_{t\ge 0}$: There are random times $0 = T^z_0
< T^z_1 < T^z_2 < \dots$ such that with $Y^z_i \coloneq X^z_{T^z_i} -
X^z_{T^z_{i-1}}$, the sequence of space-time increments along the
regeneration times $T^z_i$
\begin{align} 
  \label{eq:regenseq}
  \l(Y^z_i, T^z_i - T^z_{i-1}\r)_{i=1,2,\dots} \quad 
  \text{is i.i.d.\ under } \Pr(\cdot \mid B_z)
\end{align}
with $\widetilde\Pr_z\l( |Y^z_1| \ge n \r), \widetilde\Pr_z\l( T^z_1-T^z_0 \ge n \r) \le C e^{-c n}$ 
(see \cite[Lemma~2.5]{BCDG13}). 
By symmetry, $\widetilde\E_z[Y^z_1 ] = 0$. In fact, \eqref{eq:regenseq} already yields an (annealed)
central limit theorem with limit variance
\begin{align} 
  \label{eq:sigmasq.formula}
  v = \frac{\widetilde\E_z[(Y^z_1)^2]}{\widetilde\E_z[ T^z_1]} \in (0,\infty),
\end{align}
see \cite[Remark~1.2]{BCDG13}.
(In \cite{BCDG13}, all this is formulated for $z=(0,0)$ but by
shift-invariance of the joint distribution of $\omega$ and
$\widetilde{\omega}$, it holds for any $z \in \Z \times \Z$ and in
particular 
$v$ in \eqref{eq:sigmasq.formula} does not depend on $z$.) 
\medskip

For $z_1=(x_1,0), z_2=(x_2,0) \in \Z \times \Z$ consider the 
\emph{simultaneous regeneration times} $0=T^\mathrm{sim}_0<T^\mathrm{sim}_1<T^\mathrm{sim}_2<\dots$ 
for $X^{z_1}$ and $X^{z_2}$, defined via 
\begin{align} 
  \l\{ T^\mathrm{sim}_i : i \in \N_0 \r\} = 
  \l\{ T^{z_1}_j : j \in \N_0 \r\} \cap \l\{ T^{z_2}_{j'} : j' \in \N_0 \r\} 
\end{align}
(we have $T^\mathrm{sim}_i<\infty$ a.s.\ for all $i$, see \cite[Section~3.1]{BCDG13}).
In our notation we suppress the dependence of $T^\mathrm{sim}_i$ on the starting points $z_1, z_2$.
\smallskip

Write $\widehat{X}^{z_1}_\ell \coloneq X^{z_1}_{T^\mathrm{sim}_\ell},
\widehat{X}^{z_2}_\ell \coloneq X^{z_2}_{T^\mathrm{sim}_\ell}$,
$\ell\in\N_0$.  In \cite[Section~3.1]{BCDG13} it is shown (with a
slightly different notation, see also \cite[Remark~3.3]{BCDG13}) that
the sequence of pairs of path increments between the simultaneous
regeneration times,
\begin{align} 
  \l( \Big( X^{z_1}_{t+T^\mathrm{sim}_{\ell-1}}-\widehat{X}^{z_1}_{\ell-1}\Big)_{0 \le t \le T^\mathrm{sim}_\ell - T^\mathrm{sim}_{\ell-1}}, 
  \Big( X^{z_2}_{t+T^\mathrm{sim}_{\ell-1}}-\widehat{X}^{z_2}_{\ell-1}\Big)_{0 \le t \le T^\mathrm{sim}_\ell - T^\mathrm{sim}_{\ell-1}}, 
  T^\mathrm{sim}_\ell - T^\mathrm{sim}_{\ell-1} \r)_{\ell\in\N}
\end{align}
forms a Markov chain under $\widetilde\Pr_{z_1,z_2}$, see \cite[Lemma~3.2]{BCDG13}. 
Furthermore, the transition probabilities depend only on the current positions $\big(\widehat{X}^{z_1}_\ell, \widehat{X}^{z_2}_\ell\big)$
and the increments between simultaneous regeneration times have uniformly exponentially bounded tails, 
\begin{align} 
  \label{eq:Tsim.tailbound}
  \widetilde\Pr_{z_1,z_2}\Big( T^\mathrm{sim}_\ell - T^\mathrm{sim}_{\ell-1} > n \Big) 
  \le C e^{-c n} \quad 
  \text{for all } n\in\N, \, z_1=(x_1,0), z_2=(x_2,0) \in \Z\times\Z
\end{align}
(see \cite[Lemma~3.1]{BCDG13}). In particular, 
$\big(\widehat{X}^{z_1}_\ell, \widehat{X}^{z_2}_\ell\big)_{\ell \in\N_0}$ is itself a Markov 
chain under $\widetilde\Pr_{z_1,z_2}$ and -- by shift invariance of the 
joint distribution of $\omega$ and $\widetilde{\omega}$ -- its transition matrix is invariant under 
simultaneous shifts in both coordinates, i.e.\ 
\begin{align} 
  & \widetilde\Pr_{z_1,z_2}\Big( \widehat{X}^{z_1}_{\ell+1} = y_1', \widehat{X}^{z_2}_{\ell+1} = y_2' \, \Big| \, 
    \widehat{X}^{z_1}_\ell = y_1, \widehat{X}^{z_2}_\ell = y_2 \Big) \notag \\
  & = \widetilde\Pr_{z_1,z_2}\Big( \widehat{X}^{z_1}_{\ell+1} = y_1'+y, \widehat{X}^{z_2}_{\ell+1} = y_2'+y \, \Big| \, 
    \widehat{X}^{z_1}_\ell = y_1+y, \widehat{X}^{z_2}_\ell = y_2+y \Big)
\end{align}
for all $y_1,y_1',y_2,y_2',y \in \Z$. Thus $\widehat{D}^{z_1,z_2}_\ell
:= \widehat{X}^{z_1}_\ell-\widehat{X}^{z_2}_\ell$, the difference of
the two walks along simultaneous regeneration times, forms also a
Markov chain; we denote its transition matrix by
$\widehat{\Psi}^\mathrm{joint}_\mathrm{diff}$ 
(in the notation of \cite[Lemma~3.3]{BCDG13}, we have $\widehat{\Psi}^\mathrm{joint}_\mathrm{diff}(x,y) 
= \sum_{z \in \Z} \widehat{\Psi}^\mathrm{joint}\big((x,0), (z+y,z)\big)$). 

Because $X^{z_1}$ and $X^{z_2}$ have bounded increments, \eqref{eq:Tsim.tailbound} implies 
an exponential tail bound for jump sizes under $\widehat{\Psi}^\mathrm{joint}_\mathrm{diff}$: 
\begin{align} 
  \label{eq:Psijointtail}
  \widehat{\Psi}^\mathrm{joint}_\mathrm{diff}(x,y) \le C e^{-c|y-x|} \quad \text{for } x,y \in \Z.
\end{align}

One can implement the same construction when the two walks $X^{z_1}$
and $X^{z_2}$ move independently on independent copies of the oriented
percolation cluster (formally, let $\omega'$ be an independent copy of
$\omega$ and $\widetilde\omega'$ an independent copy of
$\widetilde\omega$, then construct $X^{z_2}$ by using $\omega'$ and
$\widetilde\omega'$ in \eqref{eq:Phidef} and \eqref{eq:defXdynamics};
we condition now on $z_1 \mathop{\to}^\omega \infty$ and $z_2
\mathop{\to}^{\omega'} \infty$).
Then $(\widehat{D}^{z_1,z_2}_\ell)_\ell$ is again a Markov chain on $\Z$, we denote 
its transition probability matrix in this case by $\widehat{\Psi}^\mathrm{ind}_\mathrm{diff}$. 
In fact, $\widehat{\Psi}^\mathrm{ind}_\mathrm{diff}$ is irreducible, symmetric and spatially 
homogeneous
(i.e., $\widehat{D}^{z_1,z_2}$ is now a symmetric random walk) 
with exponentially bounded tails
\begin{align}
  \widehat{\Psi}^\mathrm{ind}_\mathrm{diff}(x,y) = \widehat{\Psi}^\mathrm{ind}_\mathrm{diff}(y,x) 
  = \widehat{\Psi}^\mathrm{ind}_\mathrm{diff}(0,y-x) \le C e^{-c |y-x|} \quad 
  \text{for } x,y \in \Z ,
\end{align}
see \cite{BCDG13}, Section~3.1, especially the discussion after Remark~3.3.
\smallskip

Using a coupling construction and space-time mixing properties of the percolation cluster, 
one finds the following lemma. 
\begin{lemma}[{\cite[Lemma~3.4]{BCDG13}}]
  \label{lemma:BCDG.Lemma3.4}
  We have ($\norm{\cdot}_\mathrm{TV}$ denotes total variation distance)
  \begin{align} 
    \Norm{\widehat{\Psi}^\mathrm{joint}_\mathrm{diff}(x,\cdot) - \widehat{\Psi}^\mathrm{ind}_\mathrm{diff}(x,\cdot)}_\mathrm{TV} 
    \leq C e^{-c|x|} \quad 
    \text{for all } x \in \Z.
  \end{align}
\end{lemma}

\begin{remark} 
  \label{rem:Tsim.tailbound.m>2}
  One can in complete analogy to the construction for $m=2$ walks
  consider joint regeneration times for any number $m \ge 2$ of walks
  $X^{(x_1,n_1)}, \dots, X^{(x_m,n_m)}$ (obviously, joint regeneration
  can then only occur after ``real'' time
  $\max\{n_1,n_2,\dots,$ $n_m\}$). In fact, in
  Section~\ref{subsect:checkingT1} we will consider
  the case $m=5$.

  Following the construction in \cite[Section~3]{BCDG13}, one
  obtains that a tail bound for increments between joint regeneration
  times analogous to \eqref{eq:Tsim.tailbound} also holds in this
  case.
\end{remark}

\subsection{A bound on the meeting time for two walks on the cluster}

Let $T^{(z_1, z_2)}_{\mathrm{meet}} \coloneqq \inf\{ n \ge 0 : X^{(z_1)}_n = X^{(z_2)}_n\}$ (with the usual convention 
$\inf\emptyset = +\infty$).

\begin{lemma} 
  \label{main1}
  There is $C=C(p)<\infty$ such that
  \begin{align}
    \label{eq:tailbounds} 
    \widetilde\Pr_{z_1,z_2}(T^{(z_1, z_2)}_{\mathrm{meet}} > n) \le C \frac{|z_1-z_2|}{\sqrt{n}}  
    \quad \text{for } z_1, z_2 \in \Z, \, n \in \N .
  \end{align}
  In particular,
  $\widetilde\Pr_{z_1,z_2}(T^{(z_1, z_2)}_{\mathrm{meet}} < \infty) = 1$
  and hence also $P_\omega(T^{(z_1, z_2)}_{\mathrm{meet}} < \infty) = 1$
  for $\Pr$-almost all $\omega \in B_{(z_1,0)} \cap B_{(z_2,0)}$.
\end{lemma}
Instead of conditioning on $(z_1,0), (z_2,0) \in \CC$ in \eqref{eq:tailbounds}, we could 
also pick the ``nearest'' connected sites (say, on the left, as in \eqref{next point left}) 
without changing the statement.
\medskip

We are interested in collision events of two directed random walks
$X^{(z_1)}, X^{(z_2)}$ moving on the same space-time cluster $\CC$,
i.e., we ask that the two walks are at the \emph{same time} at the
\emph{same site}. Lemma~\ref{main1} tells us in particular that a
collision event between two random walks occurs almost surely in
dimension $d=1$. This is not completely obvious a priori because
``holes'' in the space-time cluster $\CC$ might at least in principle
prevent such collisions. However, the right-hand side of \eqref{eq:tailbounds} 
is -- modulo a constant -- also the correct order for the corresponding 
probability for two simple random walks on $\Z$, so that in this sense, 
the holes in the cluster do not have a strong influence.
\medskip

Fix $z_1, z_2$, put 
\begin{align} 
  \widehat{T}^{(z_1, z_2)}_{\mathrm{meet}} := \inf\big\{ \ell \in \N : \widehat{X}^{z_1}_\ell = \widehat{X}^{z_2}_\ell \big\} 
  = \inf\big\{ \ell \in \N : \widehat{D}^{z_1,z_2}_\ell = 0 \big\}. 
\end{align}
In view of \eqref{eq:Tsim.tailbound}, it suffices to establish that there is 
a constant $C=C(p) < \infty$ such that 
\begin{align}
  \label{eq:tailbounds.hat} 
  \widetilde\Pr_{z_1,z_2}\big(\widehat{T}^{(z_1, z_2)}_{\mathrm{meet}} > n\big) \le C \frac{|z_1-z_2|}{\sqrt{n}}  
  \quad \text{for } z_1, z_2 \in \Z, \, n \in \N.
\end{align}
(To pass from $\widehat{T}^{(z_1, z_2)}_{\mathrm{meet}}$ to
$T^{(z_1, z_2)}_{\mathrm{meet}}$ note that if $k(n)$ denotes the last
simultaneous regeneration time before time $n$, for sufficiently small
$c>0$, the probability of the event $\{ k(n)/n < c\}$ decays
exponentially as $n\to\infty$.)

The key ingredient for the proof of Lemma~\ref{main1} is the estimate
on the total variation error between
$\widehat{\Psi}^\mathrm{joint}_\mathrm{diff}$ and
$\widehat{\Psi}^\mathrm{ind}_\mathrm{diff}$ recalled in
Lemma~\ref{lemma:BCDG.Lemma3.4}.  Lemma~\ref{main1} is thus in a sense
a ``trivial'' instance of a so-called Lamperti problem,
$(\widehat{D}^{z_1,z_2}_\ell)_\ell$ is under $\widetilde\Pr_{z_1,z_2}$
a Markov chain that is a local perturbation of a symmetric random walk
and the drift at $x$ vanishes exponentially fast in $|x|$.  A very
fine analysis in the case of $\pm 1$-steps can be found in
\cite{Alexander:11}, see also the references there for background.
\cite{DenisovKorshunovWachtel:arXiv161201592} have established a
generalization of Alexander's results to the non-nearest neighbour
case which in particular refines \eqref{eq:tailbounds} to asymptotic
equivalence as $n\to\infty$ (see
\cite{DenisovKorshunovWachtel:arXiv161201592} Thm.~5.11 and
Lemma~5.12; cf also Cor.~5.16 for the hitting time of a point instead
of a half-interval). A recent and equally enjoyable reference on
Lamperti problems is \cite{MPW17}.  For completeness' sake we present
here a short, rough proof of the coarser estimate that suffices for
our purposes.  (More detailed arguments can also be found in
\cite[Chapter~2]{Steiber:17}.)  
\medskip

\begin{proof}[(Sketchy) proof of Lemma~\ref{main1}] 
  Write $(\widehat{D}_n)_{n\in\N_0}$ for the Markov chain on $\Z$ 
  with transition probabilities $\widehat{\Psi}^\mathrm{joint}_\mathrm{diff}$. 
  For $x\in\Z$ we will write here $\Pr^\mathrm{joint}_x$ for a probability measure under which 
  this Markov chain starts in $x$, i.e., $\Pr^\mathrm{joint}_x(\widehat{D}_0=x)=1$.
  \smallskip

  Let us first verify that there exists $x_0 > 0$, $n_0>0$ and $C$ such that 
  \begin{align}
    \label{eq:Dhit.bd1}
    \Pr^\mathrm{joint}_x\big( \tau_{x_0} > n \big) \le C \frac{|x|}{\sqrt{n}} \quad 
    \text{for all } n \ge n_0, \, x \in \Z
  \end{align} 
  where 
  \begin{align}
    \label{eq:taux0}
    \tau_{x_0} \coloneq \inf\{ n \ge 0 : |\widehat{D}_n| \le x_0 \}. 
  \end{align}
  By Lemma~\ref{lemma:BCDG.Lemma3.4} and analogous properties of $\widehat{\Psi}^\mathrm{ind}_\mathrm{diff}$ 
  we have 
  \begin{align} 
    \label{eq:Dhatcondmeanandvar}
    \big| \E[\widehat{D}_{n+1}-x | \widehat{D}_n = x] \big| \le C e^{-c x} \quad \text{and} 
    \quad \Var[\widehat{D}_{n+1} | \widehat{D}_n = x] \geq \tilde{\sigma}^2
  \end{align} 
  whenever $|x|$ is sufficiently large (for suitable $\tilde{\sigma}^2, c, C \in (0,\infty)$).

  We can find $c_1, x_0 \in (0,\infty)$ such that the function 
  \begin{align}
    \label{eq:superharmcandidate}
    f(x) = \int_0^{|x|} \exp\big( 2 e^{-c_1 y}/c_1 \big) \, dy, \quad x \in \R
  \end{align} 
  is non-negative and superharmonic for
  $\widehat{\Psi}^\mathrm{joint}_\mathrm{diff}$ in $\Z \cap
  [-x_0,x_0]^c$. This follows from Lemma~\ref{lemma:BCDG.Lemma3.4} and a Taylor
  expansion of $f$ to second order (more details are given in
  Appendix~\ref{sect:proof.superharmcandidate}).  Note that $f$ solves
  $\frac12 f''(x) + \sgn(x) e^{-c_1|x|} f'(x) = 0$ for $x \neq 0$, i.e., $f$ is
  a harmonic function for a Brownian motion with spatially
  inhomogeneous drift $\sgn(x) e^{-c_1 |x|}$. Note that $f(x)$ can in principle be
  expressed explicitly in terms of the exponential integral function 
  (see, e.g., \cite[Chapter~5]{AbramowitzStegun1964}),
  for our purposes it suffices to observe that $0 \le f(x) \le e^{2/c_1}
  |x|$. 
  \medskip 

  Thus, starting from $\widehat{D}_0=x$ with $|x|>x_0$,
  $Z_n \coloneq f(\widehat{D}_{n \wedge \tau_{x_0}})$ is a non-negative
  supermartingale with $Z_0 = f(x) \le c |x|$ and it is easy to see (cf
  \eqref{eq:Dhatcondmeanandvar}) that for some $b_1<\infty$,
  $0<b_2<\infty$
  \begin{align} 
    \big| \E[Z_{n+1} \,| \, \sigma(Z_0,\dots,Z_n)] - Z_n \big| \le b_1, \;\;\;
    \Var[Z_{n+1} \,| \, \sigma(Z_0,\dots,Z_n)] \ge b_2 \;\;
    \text{on } \{ \tau_{x_0}>n\} 
  \end{align}
  \eqref{eq:Dhit.bd1} follows then from well known tail bounds for hitting times of supermartingales (see, e.g., 
  \cite[Proposition~17.20]{LPW:09}). 
  \medskip

  Obtaining \eqref{eq:tailbounds.hat} from \eqref{eq:Dhit.bd1} is a
  fairly standard argument for irreducible Markov chains: We can find
  $M<\infty$, $\varepsilon > 0$ such that
  \begin{align} 
    \inf_{|x| \leq x_0} \Pr^\mathrm{joint}_x\big( \widehat{D} \text{ hits $0$ within at most $M$ steps without exiting 
    $[-x_0,x_0]$ before}\big) \ge \varepsilon.
  \end{align}
  Thus, starting from some $x \in [-x_0,x_0]$, the path of $\widehat{D}$
  before hitting $0$ can be decomposed into an at most geometrically
  distributed number of ``outside excursions'' out of $[-x_0,x_0]$ and
  path pieces inside $[-x_0,x_0]$, plus the final piece inside
  $[-x_0,x_0]$ when $0$ is hit for the first time.  By
  \eqref{eq:Dhit.bd1} and the (exponential) tail bounds on jumps sizes
  for $\widehat{\Psi}^\mathrm{joint}_\mathrm{diff}$, the tail of the
  length distribution of an outside excursion is bounded by $C/\sqrt{n}$
  (uniformly in $n\ge n_0$ and the starting point inside), the length
  distribution of the pieces ``inside'' has (again uniformly in the
  starting point inside) exponentially decaying tails. It is well known
  that a geometric sum of non-negative random variables with a tail
  bound of the form $C/\sqrt{n}$ again satisfies such a tail bound (with
  an enlarged $C$), thus there are $C<\infty$ and $n_0 \in \N$ with
  \begin{align}
    \label{eq:Dhit.bd2}
    \sup_{|x| \leq x_0} \Pr^\mathrm{joint}_x\big( \inf\{ \ell \ge 0 : \widehat{D}_\ell = 0 \} > n \big) \leq \frac{C}{\sqrt{n}} \quad 
    \text{for all } n \ge n_0,
  \end{align}
  see, e.g.\ the proof of Corollary~5.16 in \cite{DenisovKorshunovWachtel:arXiv161201592}. 

  Now \eqref{eq:tailbounds.hat}, with a suitably enlarged $C$, is
  for $|z_1-z_2| \leq x_0$ immediate from \eqref{eq:Dhit.bd2}, for $|z_1-z_2|
  > x_0$ it follows from \eqref{eq:Dhit.bd1} and \eqref{eq:Dhit.bd2}
  since
  \[
  \widetilde\Pr_{z_1,z_2}(\widehat{T}^{(z_1, z_2)}_{\mathrm{meet}} > n) \le 
  \Pr^\mathrm{joint}_{z_1-z_2}\big( \tau_{x_0} > n/2 \big) + 
  \sup_{|x|<x_0} \Pr^\mathrm{joint}_x\big( \inf\{ \ell \ge 0 : \widehat{D}_\ell = 0 \} > n/2 \big).
  \]
\end{proof}

\begin{remark}
  \label{rem:Dhathittingbounds}
  Put $\sigma_y \coloneq \inf\{ n \ge 0 : |\widehat{D}_n| \ge y \}$ (and recall $\tau_{x_0}$ from \eqref{eq:taux0} 
  and $x_0$ from the proof of Lemma~\ref{main1}). 
  We see from the proof of Lemma~\ref{main1} that there exist $y_0 \in \N$ and $c < \infty$ so that 
  \begin{align}
    \label{eq:Dhathittingbounds}
    \Pr^\mathrm{joint}_x(\tau_{x_0} > \sigma_y) \le c \frac{x}{y} \quad 
    \text{for all } 2x_0 < x < y \; \text{ and } \; y \ge y_0 .
  \end{align}
\end{remark}
\begin{proof} 
  With $f$ from \eqref{eq:superharmcandidate} and $\tau \coloneq \tau_{x_0} \wedge \sigma_y$, 
  the process $(f(\widehat{D}_{n \wedge \tau}))_{n\in\N_0}$ is a non-negative supermartingale (w.r.t.\ the 
  filtration generated by the Markov chain $\widehat{D}$), thus by optional stopping 
  \begin{align} 
    f(x) & \ge \E^\mathrm{joint}_x\big[ f(\widehat{D}_{\tau})\big] \notag \\
         & = \Pr^\mathrm{joint}_x(\tau_{x_0} > \sigma_y) \E^\mathrm{joint}_x\big[ f(\widehat{D}_{\tau}) \, \big| \, 
           \tau_{x_0} > \sigma_y \big] \notag \\
         & \hspace{2em} + \big( 1 - \Pr^\mathrm{joint}_x(\tau_{x_0} > \sigma_y) \big) 
           \E^\mathrm{joint}_x\big[ f(\widehat{D}_{\tau}) \, \big| \, \tau_{x_0} < \sigma_y \big] \notag \\
         & \geq \Pr^\mathrm{joint}_x(\tau_{x_0} > \sigma_y) \big( f(y) - f(x_0) \big) 
           + f(x_0).
  \end{align}
  This together with $|x| \leq f(x) \leq e^{2/c_1} |x|$ implies \eqref{eq:Dhathittingbounds}.
\end{proof}
\smallskip

The following lemma allows to control the undesirable situation that two walks
come close but then separate again and spend a long time apart before
eventually coalescing. We will need this in Section~\ref{subsect:CheckI1} below
(Checking condition\ $(I_1)$, Step~2).
\begin{lemma}
  \label{lemma:tmeet-tneartight}
  For $z_1=(x_1,t_1), z_2=(x_2,t_2) \in \Z \times \Z$ write
  \begin{align}
    \label{def:Tnearz1z2}
    T^{z_1,z_2}_{\mathrm{near}}
    & := \inf \big\{ t \in \Z, t \ge t_1 \vee t_2 : | \pi^{z_1}(t) - \pi^{z_2}(t) | \leq 1 \big\}, \\[0.5ex]
    \label{def:Tmeetz1z2}
    T^{z_1,z_2}_{\mathrm{meet}}
    & := \inf \big\{ t \in \Z, t \ge t_1 \vee t_2 : \pi^{z_1}(t) = \pi^{z_2}(t) \big\}
      \;\; \big(\geq T^{z_1,z_2}_{\mathrm{near}}\big) .
  \end{align}
  The family $\{ T^{z_1,z_2}_{\mathrm{meet}} - T^{z_1,z_2}_{\mathrm{near}} : z_1, z_2 \in \Z \times \Z \}$
  is tight, that is
  \begin{align}
    \label{eq:tmeet-tneartight}
    \lim_{M\to\infty} \sup_{z_1, z_2 \in \Z \times \Z}
    \Pr\big( T^{z_1,z_2}_{\mathrm{meet}} - T^{z_1,z_2}_{\mathrm{near}} \ge M \big) = 0.
  \end{align}
  In particular
  \begin{align}
    \label{eq:suppi1pi2tmeettnear}
    \lim_{M\to\infty} \sup_{z_1, z_2 \in \Z \times \Z}
    \Pr\Big( \sup \big\{  | \pi^{z_1}(t) - \pi^{z_2}(t) |
    : T^{z_1,z_2}_{\mathrm{near}} \leq t \leq T^{z_1,z_2}_{\mathrm{meet}} \big\} \ge M \Big) = 0.    
  \end{align}
\end{lemma}
\begin{proof}[Proof sketch]
  \eqref{eq:suppi1pi2tmeettnear} follows from \eqref{eq:tmeet-tneartight} because
  $|\pi^{z_1}(T^{z_1,z_2}_{\mathrm{near}}) - \pi^{z_2}(T^{z_1,z_2}_{\mathrm{near}})| \leq 1$ and thus
  \[
  \sup \big\{  | \pi^{z_1}(t) - \pi^{z_2}(t) |
  : T^{z_1,z_2}_{\mathrm{near}} \leq t \leq T^{z_1,z_2}_{\mathrm{meet}} \big\}
  \leq 1 + 2 \big( T^{z_1,z_2}_{\mathrm{meet}} - T^{z_1,z_2}_{\mathrm{near}} \big).
  \]
  
  For \eqref{eq:tmeet-tneartight}, consider first the case $t_1=t_2$, and
  then w.l.o.g.\ $t_1=t_2=0$, $x_1=0$. 
  Write 
  \[
  X_{n} = \pi^{z_1}(n), \;\; X'_n = \pi^{z_2}(n), \quad n\in \N_0
  \]
  for the two walks.
  The idea behind \eqref{eq:tmeet-tneartight} is that even if
  $|\pi^{z_1}(T^{z_1,z_2}_{\mathrm{near}}) - \pi^{z_2}(T^{z_1,z_2}_{\mathrm{near}})| = 1$
  and thus $T^{z_1,z_2}_{\mathrm{near}} < T^{z_1,z_2}_{\mathrm{meet}}$,
  the difference should be bounded in probability irrespective of where the two walks
  are at time $T^{z_1,z_2}_{\mathrm{near}}$ in view of Lemma~\ref{main1}.
  A little complication lies in the fact that the pair $(X_n, X'_n)_n$ is not in itself 
  a Markov chain, so we cannot simply stop at the random time $T^{z_1,z_2}_{\mathrm{near}}$
  and then apply the strong Markov property. 

  Instead, we consider the two walks along their joint regeneration times 
  $0=T^{\mathrm{sim}}_0 < T^{\mathrm{sim}}_1 < T^{\mathrm{sim}}_2 < \cdots$, 
  which yields a Markov chain $(\widehat{X}_\ell, \widehat{X}'_\ell)_{\ell\in \N_0}$ 
  (recall the discussion and notation from Section~\ref{sect:preliminaries}). 
  For $a>0$ put
  \[
  \widehat{T}(a) := 
  \inf\big\{ \ell \in \N_0 : \big| \widehat{X}_\ell -\widehat{X}'_\ell \big| \leq a \big\}.
  \]
  
  Fix $M>0$ and let $k \in \N$ be the smallest integer such that $2^k M \geq |x_2-x_1|$.
  Then 
  \begin{align*} 
    & \Pr\Big( | X_n - X'_n | \le 1 \text{ for some } n < T^{\mathrm{sim}}_{\widehat{T}(M)} \Big) \\
    & \leq \sum_{j=1}^k \Pr\Big( | X_n - X'_n | \le 1
      \text{ for some } T^{\mathrm{sim}}_{\widehat{T}(2^j M)} \leq 
      n < T^{\mathrm{sim}}_{\widehat{T}(2^{j-1} M)} \Big) \\
    & \leq \sum_{j=1}^k \Pr\big( T^{\mathrm{sim}}_\ell - T^{\mathrm{sim}}_{\ell-1} > 2^{j-2} M 
      \text{ for some } \widehat{T}(2^j M) \leq \ell < \widehat{T}(2^{j-1} M) \big) \\
    & \leq \sum_{j=1}^k \Big\{ \Pr\big( \widehat{T}(2^{j-1} M)-\widehat{T}(2^j M) > (2^j M)^3 \big) \\[-2.0ex]
    & \hspace{3.5em} + 
      \Pr\big( T^{\mathrm{sim}}_\ell - T^{\mathrm{sim}}_{\ell-1} > 2^{j-2} M
      \text{ for some } \ell \\[-1.0ex]
    & \hspace{6em} \text{ with }  \widehat{T}(2^j M)\leq \ell \leq \widehat{T}(2^j M) + (2^j M)^3 \big) \Big\} \\
    & \leq \sum_{j=1}^k \Big\{ \frac{C M2^j}{\sqrt{(2^j M)^3}} + (2^j M)^3 C \exp\big( -c 2^{j-2} M \big) \Big\} \\
    & \leq \frac{C}{M^{1/2}} \sum_{j=1}^\infty \frac{1}{2^{j/2}} +
      C M^3 \sum_{j=1}^\infty \exp\big( j \log(2) - c 2^{j-1}M \big)
      =: b(M).
  \end{align*}
  Note that $b(M)\to0$ as $M\to\infty$.
  
  We can apply the Markov property of $(\widehat{X}_\ell, \widehat{X}'_\ell)$ 
  at the stopping time $\widehat{T}(M)$ (which corresponds to time 
  $T^{\mathrm{sim}}_{\widehat{T}(M)}$ for the two walks $(X_n,X'_n)$ themselves), noting that 
  \[
  | X_{T^{\mathrm{sim}}_{\widehat{T}(M)}} - X'_{T^{\mathrm{sim}}_{\widehat{T}(M)}} | 
  = | \widehat{X}_{\widehat{T}(M)} - \widehat{X}'_{\widehat{T}(M)} | \leq M
  \]
  and thus, using shift-invariance of the joint distribution and Lemma~\ref{main1}, 
  \begin{align} 
    & \Pr\big( T^{z_1,z_2}_{\mathrm{meet}} - T^{z_1,z_2}_{\mathrm{near}} > M^5 \big) \notag \\
    & \leq \Pr\Big( | X_n - X'_n | \le 1 \text{ for some } n \leq T^{\mathrm{sim}}_{\widehat{T}(M)} \Big) 
      + \sum_{x=1}^M \widetilde{\Pr}_{0,x}\big(T^{(0,x)}_{\mathrm{meet}} > M^5 \big) \notag \\
    & \leq b(M) 
      + C \sum_{x=1}^M \frac{x}{M^{5/2}} 
      \mathop{\longrightarrow}_{M\to\infty} 0 
      \label{eq:Thit-Tclose}
  \end{align}
  and the bound in the last line holds uniformly for all $z_1,z_2 \in \Z\times\Z$.
  \smallskip

  When $t_1 \neq t_2$, say $t_2 > t_1$, we let the first walk begin at
  time $t_1$ and ``run freely'' until time $t_2$, then argue as
  above. Again, there is a slight complication because we would have
  to first look only along regeneration times, then use joint
  regeneration times as soon as the second walk ``comes into the
  picture''. This can be handled similarly as above, we do not spell
  out the details. 
\end{proof}

\subsection{Proof of Theorem~\ref{main2}}

We follow the approach developed in \cite{NewmanRavishankarSun2005} and \cite{Sun:05}.

\subsubsection{Conditions for convergence to the Brownian web}
First we introduce a little more notation which is needed to formulate
the sufficient conditions for convergence to the Brownian web from
\cite{NewmanRavishankarSun2005}.  Define
$\Lambda_{L,T} \coloneq [-L,L]\times[-T,T]\subset \R^2$. For
$x_0,t_0\in \R$ and $u,t>0$ let $R(x_0,t_0,u,t)$ be the rectangle
$[x_0-u,x_0+u]\times[t_0,t_0+t]\subset\R^2$ and define
$A_{t,u}(x_0,t_0)$ to be the 
set of $K\in
\CH$ 
which contain a path that touches both, the rectangle $R(x_0,t_0,u,t)$ and
the left or right boundary of the bigger rectangle
$R(x_0,t_0,20u,2t)$ (note $A_{t,u}(x_0,t_0) \in \CB_\CH$). 
For $a,b,t_0,t\in\R,a<b,t>0$ and $K\in \CH$, we define the
number of distinct points in $\R\times\{t_0+t\}$, which are touched by
some path in $K$ that also touches $[a,b]\times \{t_0\}$ by
\begin{align*}
  \eta(t_0,t;a,b) & \coloneq \eta_K(t_0,t;a,b) \\ 
                  & \coloneq \#\{y\in\R:\exists x\in[a,b]\text{ and a path in $K$ which} \\[-0.5ex]
                  & \hspace{6em} \text{touches both $(x,t_0)$ and $(y,t_0+t)$}\}.
\end{align*}
Similarly, let 
\begin{align*}
  &\widehat\eta(t_0,t;a,b) \coloneq \widehat\eta_K(t_0,t;a,b)\\
  & \coloneq \#\{ x \in (a,b) \text{ there is a path in $K$ which touches both $\R \times \{t_0\}$ and $(x,t_0+t)$}\}.
\end{align*}
be the number of points in $(a,b) \times \{t+t_0\}$ which are touched by some path in $K$ which started 
at time $t_0$ or before.

If $\CX$ is a $(\CH,\CB_{\CH})$-valued random variable, we define
\begin{align}
  \label{Xs-definition}
  \CX^{s^-} \text{ to be the subset of paths in $\CX$ which start before or
  at time $s$.}
\end{align} 
\medskip

Combining Theorem~1.4 and Lemma~6.1 from \cite{NewmanRavishankarSun2005}, 
we see that a family $\{\CX_n\}_n$
of $(\CH,\CB_\CH)$-valued random variables with distribution
$\{\mu_n\}_n$ converges in distribution to the standard Brownian web
$\CW$, if it satisfies the following conditions:
\begin{itemize}
\item[($I_1$)] There exist single path valued random variables $\theta_n^y\in \CX_n,\text{ for }y\in \R^2$, satisfying:\\
  for $\CD$ a deterministic countable dense subset of $\R^2$, for any
  deterministic $z_1,...,z_m \in \CD$,
  $\theta^{z_1}_n,...,\theta^{z_m}_n$ converge jointly in distribution
  as $n \ra \infty$ to coalescing Brownian motions (with unit
  diffusion constant) starting at $z_1,...,z_m$.
\item[($T_1$)] For every $u,L,T\in(0,\infty)$
  \[\widetilde{g}(t,u;L,T)\equiv t^{-1} \limsup_{n\to\infty} \underset{(x_0,t_0)\in\Lambda_{L,T}}{\sup} \mu_n(A_{t,u}(x_0,t_0))\lra 0 
  \;  \text{ as }t\ra 0^+,\]
  which is a sufficient condition for the family $\{\CX_n\}_n$ to be tight.
\item[($B_1'$)] For all $\beta>0$
  \[\limsup_{n\to\infty} \;\underset{t>\beta}{\sup}\;\underset{t_0,a\in \R}{\sup}\mu_n(\eta(t_0,t;a-\e,a+\e)>1)\lra0 \; \text{ as }\e\ra0^+.\]
\item[$(E_1')$] If $\CZ_{t_0}$ is any subsequential limit of $\{\CX_n^{t_0^-}\}_n$ for any $t_0\in\R$, then for all $t,a,b\in \R$, with $t>0$ and $a<b$, 
  \[\E[\widehat\eta_{\CZ_{t_0}}(t_0,t;a,b)]\leq\E[\widehat\eta_\CW(t_0,t;a,b)]=\frac{b-a}{\sqrt{\pi t}}.\]
\end{itemize}  
\begin{remark}   
  1.\ We consider the diffusively rescaled closure of
  $\mathbf{\Gamma}=\{\pi^z:z\in \Z\times\Z\}=\{\pi^z:z\in \CC\}$,
  which is the collection of all linearly interpolated random walk
  paths. Therefore, instead of $\CX_n$, we usually write $\CX_\delta$
  to denote the $(\CH,\CB_{\CH})$-valued random variable
  $S_{v ,\delta}\mathbf{\Gamma}$.  If we want to consider the weak
  limit of $(\CX_\delta)_{\delta>0}$ along a certain subsequence
  $(\delta_n)_n$, where $\delta_n\to 0$ as $n\to \infty$, we denote
  the random variables $S_{v, \delta_n}\mathbf{\Gamma}$ by
  $\CX_{\delta_n}$. The probability measure
  $\Pr\circ(S_{v, \delta_n}\mathbf{\Gamma})^{-1}$ on $(\CH,\CB_{\CH})$
  is denoted by $\mu_{\delta_n}$.  \smallskip

  2.\ We invoke condition $(E_1')$ because because in our model, paths
  $\pi^{z_1}$ and $\pi^{z_2}$ can cross each other without
  coalescing. In this respect, our scenario is different from that in
  \cite{SarkarSun:2013}.
\end{remark}
\subsubsection{Checking condition \texorpdfstring{$(I_1)$}{I1}}
\label{subsect:CheckI1}

Let $\CD$ be a dense countable subset of $\R^2$ and
choose distinct $y_1=(x_1,t_1),...,y_m=(x_m,t_m)\in \CD$.
Define $y_{\delta,i}:=(\left\lfloor x_i v \delta^{-1} \right\rfloor,\left\lfloor t_i\delta^{-2}\right\rfloor).$ 
Let 
\begin{align} 
  \pi_\delta^i:=S^1_{v,\delta}\circ\pi^{y_{\delta,i}}, \quad i=1,\dots,m
\end{align}
be the corresponding diffusively rescaled (and coalescing) random walks. 
In order to show that $(\pi_\delta^1,\dots,\pi_\delta^m)$ converges to a system of 
$m$ coalescing Brownian motions as $\delta \ra 0$, we will follow the strategy from 
\cite{NewmanRavishankarSun2005} and construct a suitable coupling with $m$ 
independent walks on the cluster. One could alternatively attempt to 
use the characterization of coalescing Brownian motions via a martingale 
problem, we discuss this briefly in Remark~\ref{rem:HowittWarren} below.
\medskip

We will need some auxiliary types of paths: 
Let $\widetilde{X}^{(c(y_{\delta,i}))} = (\widetilde{X}^{(c(y_{\delta,i}))}_t)_{t \in \Z+}$, $i=1,\dots,m$ be 
independent conditional on $\CC$ with transition probabilities given by \eqref{eq:defXtransprob}, 
i.e., $\widetilde{X}^{(c(y_{\delta,i}))},\dots,\widetilde{X}^{(c(y_{\delta,m}))}$ are $m$ independent walks on 
the same realization of the cluster, with $\widetilde{X}^{(c(y_{\delta,i}))}$ starting from the nearest possible starting 
point to $y_{\delta,i}$ on $\CC$ (recall $c(z)$ from \eqref{next point left}). Note that we 
can for example construct these walks as in \eqref{eq:defXdynamics} and \eqref{eq:Phidef} by using 
$m$ independent copies of $\widetilde{\omega}$. 
Let $\widetilde{\pi}^{(c(y_{\delta,i}))}$ be the extension of $\widetilde{X}^{(c(y_{\delta,i}))}$ to real times 
by linear interpolation, and denote their rescalings by 
\begin{align} 
  \label{def:pitilde.delta.i}
  \widetilde\pi_\delta^i:=S^1_{v,\delta}\circ \widetilde{\pi}^{(c(y_{\delta,i}))}, \quad i=1,\dots,m.
\end{align}
Note that $\widetilde\pi_\delta^i \mathop{=}^d \pi_\delta^i$ for every $i$ but unlike the $\pi_\delta^i$'s, 
different paths $\widetilde\pi_\delta^i$ and $\widetilde\pi_\delta^j$ with $j\neq i$ can meet 
at times $\in \delta^2 \Z$ and then separate 
again. 
\medskip

Furthermore, we need two different coalescence rules on $\Pi^m$: Under
the first rule $\Gamma_\alpha$, paths are merged when they first
coincide.  Let $((f_1,\sigma_1),\dots,(f_m,\sigma_m)) \in
\Pi^{m}$. Define
\begin{equation*}
  T^{i,j}_\alpha \coloneq \inf\{t>\sigma_i\vee\sigma_j, t \in \R : f_i(t)=f_j(t)\}.
\end{equation*}
Note that $t \in \R$ can be arbitrary, in particular
$t \not\in \delta^2\Z$ is possible.

Start with the (trivial) equivalence relation $i\sim i,\;i\not\sim j$
for all $i\neq j$ on $\{1,...,m\}$.  Define
\begin{equation*}
  \tau_\alpha \coloneq \underset{1\leq i,j\leq m, i\not\sim j}{\min}T_\alpha^{i,j}, \quad \text{ with }\min\emptyset =\infty
\end{equation*}
and 
\[\Gamma_\alpha (f_i(t)) \coloneq
\begin{cases}
  f_i(t),&\text{ if } t<\tau_\alpha\\
  f_{i^\ast}(t),&\text{ if } t\geq\tau_\alpha
\end{cases}\]
where
$i^{\ast}= \min\{j : (j\sim i)\text{ or }(j\not\sim i\text{ and }
T^{i,j}_\alpha=\tau_\alpha)\}$.
Update the equivalence relation at time $\tau_\alpha$ by assigning
$i\sim i^\ast$ (and implicitly also $i \sim i'$ for all
$i' \sim i^*$).  Iterating this procedure, we get the desired
structure of coalescing random walks. We label the successive times
$\tau_\alpha$ by $\tau_\alpha^1,...,\tau_\alpha^k$, where
$k\in\{1,...,m\}$ is the smallest index such that
$\tau_\alpha^k=\infty$ (after $k$ steps, either all paths have been
merged or no further meeting of paths occurs).  We will denote the
resulting $m$-tuple of paths by
$\Gamma_\alpha\big( (f_1,\sigma_1),\dots,(f_m,\sigma_m) \big)$.
\medskip

When we apply $\Gamma_\alpha$ to $(\pi_\delta^1,\dots,\pi_\delta^m)$
it may because of the linear interpolation happen that paths are
merged even though the underlying discrete walks did not meet.  This
is not literally the correct dynamics and is not the case for the
second coalescence rule $\Gamma_{\beta,\delta}$. 
\smallskip

$\Gamma_{\beta,\delta} : \Pi^m \to \Pi^m$ is defined analogously to
$\Gamma_\alpha$ except that we replace in the construction
$T^{i,j}_\alpha$ by
\begin{equation*}
  T^{i,j}_{\beta,\delta} \coloneq \inf\{t \in \delta^2\Z : t \geq \sigma_i\vee\sigma_j \text{ and } f_i(t) = f_j(t) \}.
\end{equation*}
\smallskip
Note that by construction 
\[
\Big( \Gamma_{\alpha}\big( \widetilde\pi_\delta^1, \dots, \widetilde\pi_\delta^m \big), \Gamma_{\beta,\delta}\big( \widetilde\pi_\delta^1, \dots, \widetilde\pi_\delta^m \big) \Big) \mathop{=}^d 
\Big( \Gamma_{\alpha}\big( \pi_\delta^1, \dots, \pi_\delta^m \big), \big( \pi_\delta^1, \dots, \pi_\delta^m \big)\Big),
\]
thus in particular
\begin{align}
  \label{eq:Gamma.beta.tilde.gleich}
  \Gamma_{\beta,\delta}\big( \widetilde\pi_\delta^1, \dots, \widetilde\pi_\delta^m \big) \mathop{=}^d 
  \big( \pi_\delta^1, \dots, \pi_\delta^m \big) = \Gamma_{\beta,\delta}\big( \pi_\delta^1, \dots, \pi_\delta^m \big)
\end{align}
and 
\begin{align}
  \label{eq:Gamma.alpha.gleich}
  \Gamma_{\alpha}\big( \widetilde\pi_\delta^1, \dots, \widetilde\pi_\delta^m \big) \mathop{=}^d 
  \Gamma_{\alpha}\big( \pi_\delta^1, \dots, \pi_\delta^m \big).
\end{align}

With our preparations, to verify condition $(I_1)$, it suffices to show: 
\begin{enumerate} 
\item Show that
  $(\widetilde\pi_\delta^1, \dots, \widetilde\pi_\delta^m)$ converges
  as $\delta \to 0$ in distribution on $\Pi^m$ to $m$ independent
  Brownian motions $(\CB^1,\dots,\CB^m)$.
\item Show that
  $\Gamma_{\alpha}\big( \pi_\delta^1, \dots, \pi_\delta^m \big)$ and
  $\Gamma_{\beta,\delta}\big( \pi_\delta^1, \dots, \pi_\delta^m \big)$
  are close with high probability as $\delta\to 0$.
\item Using Step~1 and \eqref{eq:Gamma.alpha.gleich},
  $\Gamma_{\alpha}\big( \pi_\delta^1, \dots, \pi_\delta^m \big)$
  converges in distribution to $m$ coalescing Brownian motions
  $(\CB^1_\mathrm{coal},\dots,\CB^m_\mathrm{coal}) =
  \Gamma_\alpha(\CB^1,\dots,\CB^m)$
  with the correct starting points.  Combining Step~2 and
  \eqref{eq:Gamma.beta.tilde.gleich} then yields the claim.
\end{enumerate}

\bigskip

\noindent
\underline{Step 1:} 
Let us verify that 
\begin{align} 
  \label{eq:pitildelimit}
  (\widetilde\pi_\delta^1, \dots, \widetilde\pi_\delta^m) \mathop{\lra}^d_{\delta\to 0} (\CB^1,\dots,\CB^m),
\end{align}
where $\CB^1,\dots,\CB^m$ are independent Brownian motions and $\CB^i$
starts from $y_i$.  Obviously, any limit will have the correct
starting points by construction.  To identify the limit, we
essentially apply the quenched CLT from \cite{BCDG13} $m$ times, but
we have to be a little careful because the rescaled starting points
$y_{\delta,i}$ might be inside a ``hole'' of the cluster $\CC$.
\smallskip

Using \cite[Theorem 1.1, Remark 1.5]{BCDG13} we know that for every
$(x,n)\in\Z\times\Z$ the diffusively rescaled random walk
$\pi_\delta^{(x,n)}$ converges weakly under
$\Pr(\cdot\vert B_{(x,n)})$ to a Brownian motion, where $B_{(x,n)}$ is
the event that $(x,n)$ is connected to infinity. Define $G_{(x,n)}$ to
be the event that the quenched functional central limit theorem holds
for a path starting in $(x,n)$.  \cite[Theorem 1.1, Theorem
1.4]{BCDG13} yields $\Pr(G_{(x,n)} \vert B_{(x,n)})=1$, hence
\[
G \coloneq \bigcap_{(x,n)\in\Z^2} \Big( G_{(x,n)}\cup (B_{(x,n)})^c
\Big)
\] 
satisfies $\Pr(G)=1$ since the complement is a countable union of null
sets. Thus up to a $\Pr$-null set either $(x,n)\in \Z\times\Z$ is not
connected to infinity or the quenched functional central limit theorem
holds in $(x,n)$.  Keeping this in mind, in order to prove the claim
of Step~1, is remains to show that
\begin{equation}
  \label{eq:ydeltainahanxi}
  \frac{c(y_{\delta, i})\delta}{ v } 
  \mathop{\longrightarrow}_{\delta \downarrow\infty} x_i \quad \text{in probability},
\end{equation}
where $c((x,n))=\max\{y \leq x:(y,n)\in \CC\}$ as defined in 
\eqref{next point left}.
\smallskip

According to \cite[Section 10, in particular Eq.~(5) on p.~1029]{Durrett:84} we know that there exist $K,C>0$ such that
\begin{align} 
  \label{eq:hole.probab.bd}
  \Pr\l(\vert x-c((x,m))\vert\geq K\log(1/\delta)\r)\leq C \delta^{2} \;\; \text{ for all }(x,m)\in\Z\times\Z\; \text{ and } \delta \in (0,1).
\end{align}
The bound \eqref{eq:hole.probab.bd} on the probability of holes of
order $\approx \log(1/\delta)$ to occur implies
\[\Pr\l(\l|x_i-\frac{c(y_{\delta,i})\delta}{v}\r|>\e\r)=\Pr\l(\l|x_i v \delta^{-1}-c(y_{\delta,i})\r|>\frac{\e v}{\delta}\r)\lra0
\quad\text{as }\delta\downarrow0\]
for every $\e > 0$ and $i=1,\dots,m$, from which 
\eqref{eq:ydeltainahanxi} and thus \eqref{eq:pitildelimit} follow.
\medskip

\noindent
\underline{Step 2:} Let us write
$(\pi^1_{\delta,\alpha},\dots,\pi^m_{\delta,\alpha}) =
\Gamma_{\alpha}\big( \pi_\delta^1, \dots, \pi_\delta^m \big)$
and recall from \eqref{eq:Gamma.beta.tilde.gleich} that
$( \pi_\delta^1, \dots, \pi_\delta^m) = \Gamma_{\beta,\delta}\big(
\pi_\delta^1, \dots, \pi_\delta^m \big)$.
We metrize $\Pi^m$ with the product metric $d^{*m}$ based on
$d(\cdot,\cdot)$ from \eqref{def:dmetric}. 
\smallskip

We claim that for every $\e>0$, 
\begin{align}
  \label{eq:Step 2}
  \Pr\left( d^{*m}\Big( \big( \pi^1_{\delta,\alpha},\dots,\pi^m_{\delta,\alpha} \big), \, 
  \big( \pi_\delta^1, \dots, \pi_\delta^m \big) \Big) \geq \e \right) 
  \mathop{\longrightarrow}_{\delta \to 0} 0 
\end{align}
(comparing with the definition of $d$ in \eqref{def:dmetric}, we leave the dependence on the starting 
times implicit here).

Define a new metric 
\[
\bar{d}((f,\sigma),(g,\sigma')) \coloneq \vert\sigma-\sigma'\vert \vee
\underset{t\in \R}{\sup}\vert f(t\vee \sigma)-g(t\vee \sigma')\vert
\] 
on $\Pi$ and analogously $\bar{d}^{*m}$ on $\Pi^m$.  We have
$d((f_1,t_1),(f_2,t_2))\leq \bar d((f_1,t_1),(f_2,t_2))$ for all
$(f_1,t_1),(f_2,t_2)\in \Pi$, since $\tanh(\cdot)$ is Lipschitz
continuous with Lipschitz constant one. Therefore in order to prove
\eqref{eq:Step 2} its enough to show that
\begin{align}
  \label{eq:Step 2b}
  & \Pr\l(\bar d^{\ast m}\l[\l((\pi^1_{\delta,\alpha}, \lfloor \delta^{-2} t_1\rfloor), \dots, (\pi^m_{\delta,\alpha},\lfloor \delta^{-2} t_m\rfloor)\r), 
    \l((\pi^1_\delta, \lfloor \delta^{-2} t_1\rfloor), \dots, (\pi^m_\delta,\lfloor \delta^{-2} t_m\rfloor)\r)\r]\geq 
    \e \r) \notag \\
  & \hspace{33.3em}  
    \mathop{\longrightarrow}_{\delta \to 0} 0 .
\end{align}

We prove \eqref{eq:Step 2b} by induction over $m$. 
\smallskip

Let $m=2$. Since $\pi^1_{\delta,\alpha} = \pi^1_\delta$ by construction we get that
\begin{align}
  & \bar d^{\ast 2}
    \l[ \l((\pi^1_{\delta,\alpha}, \lfloor \delta^{-2} t_1\rfloor),(\pi^2_{\delta,\alpha},\lfloor \delta^{-2} t_2\rfloor)\r), 
    \l((\pi^1_\delta, \lfloor \delta^{-2} t_1\rfloor),(\pi^2_\delta,\lfloor \delta^{-2} t_2\rfloor)\r)
    \r] \notag \\
  & \hspace{1em} = \bar d\l[ (\pi^2_{\delta,\alpha},\lfloor \delta^{-2} t_2\rfloor), (\pi^2_\delta,\lfloor \delta^{-2} t_2\rfloor)\r] \notag \\
  & \hspace{1em} \leq
    \delta \sup \big\{  | \pi^{y_{\delta,1}}(t) - \pi^{y_{\delta,2}}(t) |
    : T^{y_{\delta,1},y_{\delta,2}}_{\mathrm{near}} \leq t \leq T^{y_{\delta,1},y_{\delta,2}}_{\mathrm{meet}} \big\}
    \label{eq:distm=2}
\end{align}
(recall $T^{y_{\delta,1},y_{\delta,2}}_{\mathrm{near}}$ from
\eqref{def:Tnearz1z2} and
$T^{y_{\delta,1},y_{\delta,2}}_{\mathrm{meet}}$ from
\eqref{def:Tmeetz1z2}).  The bound in \eqref{eq:distm=2} holds because
$\pi^z$'s are linear interpolations of discrete walks with steps from
$\{-1,0,1\}$ and by definition of the merging rule $\Gamma_\alpha$,
$\pi^2_{\delta,\alpha}(t) = \pi^2_\delta(t)$ for
$t<T^{y_{\delta,1},y_{\delta,2}}_{\mathrm{near}}$.  \eqref{eq:distm=2}
and \eqref{eq:suppi1pi2tmeettnear} from
Lemma~\ref{lemma:tmeet-tneartight} imply \eqref{eq:Step 2b} for $m=2$.
\medskip

Now let $m>2$. Here, we can argue essentially analogously to \cite[p.~45]{NewmanRavishankarSun2005}.
There are two possibilities for the event in \eqref{eq:Step 2b} to occur.\\
The first possibility is that a ``wrong'' ($\alpha$-)coalescing event
occurs, which means that for some $k$ and $i<j$ a path $\pi^l_1$,
$l<i$ coalesces or changes its relative order with $\pi^i_1$ after
time $\tau_\alpha^k=T_\alpha^{i,j}$ and before time
$T^{i,j}_\beta$ (where there is then no need for $\pi^l_1$ and $\pi^j_1$ to coalesce ``soon'' since 
their paths did not cross before). 
Let us consider this case. 

Using Step~1 (see also \cite[Theorem 1.3,
Remark 1.5, Remark 3.11]{BCDG13}) and \eqref{eq:Gamma.alpha.gleich} together 
with the fact that $\CL(\CB^1,...,\CB^m)$ has full measure on the set of
continuity points of the mapping $\Gamma_\alpha$, $\Gamma_{\alpha}\big(
\widetilde{\pi}_\delta^1, \dots, \widetilde{\pi}_\delta^m \big)$ converges in distribution on $\Pi^m$ to
$m$ coalescing Brownian motions
\[
(\CB^1_\mathrm{coal},\dots,\CB^m_\mathrm{coal}) = \Gamma_\alpha(\CB^1,\dots,\CB^m)
\]
with the correct starting points.  Write
$\{T^{i,j}_{\alpha,\delta}\}_{1\le i,j\le m}$ for the coalescence
times of
$(\widetilde{\pi}^1_{\delta,\alpha},\dots,\widetilde{\pi}^m_{\delta,\alpha})
= \Gamma_{\alpha}\big( \widetilde{\pi}_\delta^1, \dots,
\widetilde{\pi}_\delta^m \big)$
and $\{T^{i,j}_{\beta,\delta}\}_{1\le i,j\le m}$ for the coalescence
times of
$(\widetilde{\pi}^1_{\delta,\beta},\dots,\widetilde{\pi}^m_{\delta,\beta})
= \Gamma_{\beta}\big( \widetilde{\pi}_\delta^1, \dots,
\widetilde{\pi}_\delta^m \big)$.
We thus obtain for all $i \neq j \le m$
\begin{align} 
  \l(T^{i,j}_{\alpha,\delta}\r)_{1 \le i \neq j \le m} \mathop{\lra}^d_{\delta\to0} \l(\tau^{i,j}\r)_{1 \le i \neq j \le m} 
\end{align}
where $\tau^{i,j}$ is the coalescence time (and indeed also the first
crossing time) of $\CB^i_\mathrm{coal}$ and $\CB^j_\mathrm{coal}$.
Note that almost surely,
$(\CB^1_\mathrm{coal},\dots,\CB^m_\mathrm{coal}) =
\Gamma_\alpha(\CB^1,\dots,\CB^m)$
arises via $m-1$ distinct coalescence events at a.s.\ distinct times.
\smallskip

Furthermore, Lemma~\ref{lemma:tmeet-tneartight} shows that for every $\e>0$, the events
\[
A_\delta(\e) := \bigcap_{1 \le i < j \leq m} \Big\{ T^{y_{\delta,i},y_{\delta,j}}_{\mathrm{meet}} - T^{y_{\delta,i},y_{\delta,j}}_{\mathrm{near}} < \frac{\e}{\delta^2} \Big\}
\]
satisfy $\lim_{\delta \downarrow 0} \Pr(A_\delta(\e)) = 1$.
On the event
\[
A_\delta(\e) \cap \bigg\{ \inf\Big( \big\{ |T^{i,j}_{\alpha,\delta}-T^{i',j'}_{\alpha,\delta}| :
1 \leq i,j,i',j' \leq m, (i,j) \neq (i',j')
\big\} \, \setminus \, \{0\} \Big) > 2\e \bigg\}
\]
we have
\[
\max_{1\leq i<j\leq m} |T_{\alpha,\delta}^{i,j}-T_{\beta,\delta}^{i,j}| \leq \e.
\]
Since $\e>0$ is arbitrary, we have in fact 
\[ \max_{1\leq i<j\leq m} |T_{\alpha,\delta}^{i,j}-T_{\beta,\delta}^{i,j}| \mathop{\longrightarrow}_{\delta\to0} 0 \quad 
\text{in probability.}
\]
But then the probability of a ``wrong'' coalescing event
tends to zero, since all the crossing times of the Brownian motions
are a.s.\ distinct.
\smallskip

The second possibility for the event in \eqref{eq:Step 2b} to occur is
that there is ``too much'' time between the crossing and the
coalescence. ``Too much'' time means there is a positive probability
that at least one pair of the random walks needs more than
$\e/\delta^2$ steps to coalesce after their paths crossed, for some
$\e>0$, which would allow
$\max_{1 \le i \neq j \neq m} \sup_t | \pi^i_{\delta,\alpha}(t) -
\pi^i_\delta(t)|$
to remain ``macroscopic''.  This is ruled out by an argument similar
to the one above, note that again by
Lemma~\ref{lemma:tmeet-tneartight}, the events
\[
A'_\delta(\e) := \bigcap_{1 \le i < j \leq m} \Big\{ 
\sup \big\{  | \pi^{\delta,i}(t) - \pi^{y_{\delta,j}}(t) |
: T^{y_{\delta,i},y_{\delta,j}}_{\mathrm{near}} \leq t \leq T^{y_{\delta,i},y_{\delta,j}}_{\mathrm{meet}} \big\} < \frac{\e}{\delta} \Big\}
\]
satisfy $\lim_{\delta \downarrow 0} \Pr(A'_\delta(\e)) = 1$ for every $\e>0$.
Thus, the proof of \eqref{eq:Step 2b} for $m>2$ is completed.
\medskip

\noindent
\underline{Step 3 (Verification of ($I_1$)):}
Combine \eqref{eq:pitildelimit}, \eqref{eq:Gamma.alpha.gleich} and \eqref{eq:Step 2} 
to see that 
\begin{equation}
  \label{eq:I1conclusion}
  \big( \pi_\delta^1, \dots, \pi_\delta^m \big) \mathop{\longrightarrow}^d_{\delta\to0} 
  (\CB^1_\mathrm{coal},\dots,\CB^m_\mathrm{coal}).
\end{equation}
\smallskip

\begin{remark} 
  \label{rem:HowittWarren}
  An alternative route to \eqref{eq:I1conclusion} would be to use the
  characterization of the law of coalescing Brownian motions (viewed as
  the special case of $\theta$-sticky Brownian motions with $\theta=0$)
  as the unique solution of a martingale problem from
  \cite[Theorem~76]{Howitt:2007}.  See also
  \cite[Theorem~2.1]{HowittWarren:AOP2009} and the discussion in
  \cite[Section~5]{SchertzerSunSwart2015},
  \cite[Appendix~A]{SchertzerSunSwart:MemAMS2014} as well as
  \cite[Appendix~A]{SchertzerSun:arxiv2018}.  In fact, this would
  require to check that for any weak limit point
  $(\tilde{\CB}^1,\dots,\tilde{\CB}^m)$ of
  $(\pi_{\delta_n}^1, \dots, \pi_{\delta_n}^m)$ with $\delta_n\to 0$,
  the following holds: Let $\mathcal{F} = (\mathcal{F}_t)_{t \in \R}$
  with
  $\mathcal{F}_t = \sigma((\tilde{\CB}^i(s \wedge t))_{s \ge t_i},
  \text{for $i$ s.th. } t_i \le t)$ be the joint filtration generated by
  $\tilde{\CB}^1,\dots,\tilde{\CB}^m$.  Then 1.\ each $\tilde{\CB}^i$ is
  an $\mathcal{F}$-Brownian motion starting from space-time point
  $y_i=(x_i,t_i)$, and 2.\ each pair $(\tilde{\CB}^i, \tilde{\CB}^j)$,
  $i \neq j$ is distributed as a pair of coalescing Brownian motions
  (w.r.t.\ the filtration $\mathcal{F}$).

  The fact that each $\tilde{\CB}^i$ individually is a Brownian motion
  follows immediately from the central limit theorem proved in
  \cite{BCDG13} together with Step~1 above and the fact that
  $(\tilde{\CB}^i, \tilde{\CB}^j)$ are coalescing Brownian motions was
  checked in Step~2, case $m=2$ above. However, in our set-up it appears
  quite cumbersome to verify directly that these properties also hold
  with respect to the larger joint filtration $\mathcal{F}$.  The
  natural way to such a result is to consider
  $(\pi^{y_{\delta,1}}, \dots, \pi^{y_{\delta,m}})$ along joint
  regeneration times (cf Remark~\ref{rem:Tsim.tailbound.m>2}). This
  yields a Markov chain on $\Z^m$, then one would need a suitable
  $m$-coordinate analogue of Lemma~\ref{lemma:BCDG.Lemma3.4} and
  therewith implement a martingale plus remainder term decomposition of
  the coordinates of this chain analogous to the construction in
  \cite[Section~3.4]{BCDG13} to conclude.  In our view, spelling out the
  details would be more laborious than the approach discussed above.  On
  the other hand, using \eqref{eq:I1conclusion} we can conclude that
  properties 1.\ and 2.\ discussed above do hold.
\end{remark}

\subsubsection{Checking condition \texorpdfstring{$(T_1)$}{T1}}
\label{subsect:checkingT1}
Let $A^+_{t,u}(x_0,t_0)$ be the set of $K\in\CH$ which contain a path
touching both $R(x_0,t_0,u,t)$ and the right boundary of the bigger
rectangle $R(x_0,t_0,20u,2t)$. Similarly we define
$A^-_{t,u}(x_0,t_0)$ as the event that the path hits the left boundary
of the bigger rectangle. If a variable is diffusively scaled we will
add a ``$\sim$'' to it, where $\tilde t=t\delta^{-2}$ if $t$ is a time
variable and $\tilde x=v x \delta^{-1}$ if $x$ is a space-variable.
In order to verify condition $(T_1)$ it is enough to show that for
every $u\in(0,\infty)$
\begin{align}
  \label{eq:condT1+}
  t^{-1} \limsup_{\delta\to 0} \; \mu_1(A^+_{\tilde t,\tilde
  u}(0,0))\lra 0\; \text{ as }t\ra 0^+, 
\end{align}
where we omitted the sup over $(x_0,t_0)$ from condition $(T_1)$ because of
the spatial invariance of $\mu_1=\Pr\circ(S_{v, 1}\mathbf{\Gamma})^{-1}$. 
\eqref{eq:condT1+} implies $(T_1)$ since $A_{t,u}(x_0,t_0)=A^+_{t,u}(x_0,t_0)\cup
A^-_{t,u}(x_0,t_0)$ and $\mu_1(A^-_{\tilde t,\tilde u}(0,0))$ can be estimated 
completely analogously (in fact, we even have 
$\mu_1(A^-_{\tilde t,\tilde u}(0,0)) = \mu_1(A^+_{\tilde t,\tilde u}(0,0))$ 
by symmetry). 
\medskip

We will show that for every fixed $u>0$, $\limsup_{\delta\to0}
\mu_1(A^+_{\tilde t,\tilde u}(0,0))$ is in $\textbf{o}(t)$. Let $u>0$
and define $x_{1,\delta}:=\left\lfloor 3\tilde u \right\rfloor,
x_{2,\delta}:=\left\lfloor 8\tilde u \right\rfloor,
x_{3,\delta}:=\left\lfloor 13\tilde u \right\rfloor$ and
$x_{4,\delta}:=\left\lfloor 18\tilde u \right\rfloor$ 
with $\tilde u = v u \delta^{-1}$. We are
interested in the paths
$\pi^{x_{i,\delta}}:=\pi^{(x_{i,\delta},0)},\;i=1,2,3,4$.  \smallskip

We denote by $B_i$ the event that $\pi^{x_{i,\delta}}$ stays within
distance $\tilde u$ of $x_{i,\delta}$ up to time $2\tilde t$. For a
fixed $(x,m)\in R(\tilde u,\tilde t):=R(0,0,\tilde u,\tilde t)$ denote
the times when the random walker $\pi^{(x,m)}$ first exceeds
$5\tilde u$, $10\tilde u$, $15\tilde u$ and $20\tilde u$ by
$\tau_1^{(x,m)},\;\tau_2^{(x,m)},\;\tau_3^{(x,m)}$ and
$\tau_4^{(x,m)}$. Furthermore define $\tau_0^{(x,m)}=0$ and
$\tau_5^{(x,m)}=2\tilde t$. Denote by $C_i(x,m)$ the event that
$\pi^{(x,m)}$ does not coalesce with $\pi^{x_{i,\delta}}$ before time
$2\tilde t$.
\begin{figure}
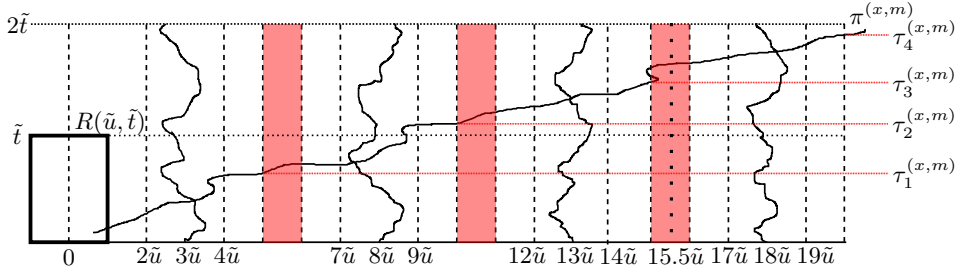

  \begin{center}
    \small 
    \def\svgwidth{12.8cm}
    \begingroup%
    \makeatletter%
    \providecommand\color[2][]{%
      \errmessage{(Inkscape) Color is used for the text in Inkscape, but the package 'color.sty' is not loaded}%
      \renewcommand\color[2][]{}%
    }%
    \providecommand\transparent[1]{%
      \errmessage{(Inkscape) Transparency is used (non-zero) for the text in Inkscape, but the package 'transparent.sty' is not loaded}%
      \renewcommand\transparent[1]{}%
    }%
    \providecommand\rotatebox[2]{#2}%
    \ifx\svgwidth\undefined%
    \setlength{\unitlength}{796.51612549bp}%
    \ifx\svgscale\undefined%
    \relax%
    \else%
    \setlength{\unitlength}{\unitlength * \real{\svgscale}}%
    \fi%
    \else%
    \setlength{\unitlength}{\svgwidth}%
    \fi%
    \global\let\svgwidth\undefined%
    \global\let\svgscale\undefined%
    \makeatother%
%
    \endgroup%
  \end{center}
  \caption{An illustration of (part of) the event
    $\mathbf{\Gamma} \in A^+_{\tilde{t},\tilde{u}}(0,0) \cap
    \bigcap_{i=1}^4B_i$ }
  \label{fig:badcrossing}
\end{figure}
We assume that $\tilde t \in \Z$, if not we replace $\tilde t$ by $\left\lceil \tilde t\right\rceil$. 
We estimate the probability in \eqref{eq:condT1+} 
in the following way (see Figure~\ref{fig:badcrossing}):
\begin{align}
  \mu_1\l(A_{\tilde t,\tilde u}^+(0,0)\r)&\leq \mu_1\l(\bigcup_{i=1}^4 B_i^c\r)\tag{$\ast$}\\
                                         &\hspace{1em} 
                                           +\mu_1\l( \bigcap_{i=1}^4B_i \: \cap \bigcup_{(x,m)\in R(\tilde u, \tilde t)} \bigg( \bigcap_{i=1}^4 C_i(x,m) \cap \{\tau_4^{(x,m)}<2\tilde t\} \bigg) \r) 
                                           \tag{$\ast\ast$}
\end{align}
We estimate the terms $(\ast)$ and $(\ast\ast)$ separately. 
We have 
\begin{align*}
  \limsup_{\delta\to 0} \mu_1\l(\bigcup_{i=1}^4 B_i^c\r)
  &\leq 
    4 \limsup_{\delta\to 0} \mu_1(B_1^c)  \\
  & = 4\Pr\l(\sup_{s\in[0,t]}\vert B_s\vert>u\r) \le 16 e^{-\tfrac{u^2}{2t}}\in \textbf{o}(t)
    \;\; \text{as }t\downarrow 0
\end{align*}
where $B$ is a standard Brownian motion.

The second term $(\ast\ast)$ can be estimated as follows
\begin{align*}
  (\ast\ast)\leq \sum_{\substack{x\in\l[-\tilde u,\tilde u\r]\cap \Z\\m\in\l[0,\tilde t\r]\cap \Z}}
  \mu_1\l(\bigcap_{i=1}^4B_i \cap \bigcap_{i=1}^4C_i(x,m) \cap \{\tau_4^{(x,m)}<2\tilde t\}\r)
\end{align*}

Now we change our point of view on the problem. From now on we come
back to the discrete structure and are only interested in the values
of the random walk path at simultaneous regeneration times
$T^\mathrm{sim}_j$ (of the five random walks), recall the discussion
in Section~\ref{sect:preliminaries} and especially
Remark~\ref{rem:Tsim.tailbound.m>2}.

Denote by $\theta_i$ the first simultaneous regeneration
time when $\pi^{(x,m)}(n)-\pi^{x_{i,\delta}}(n)>0$. Furthermore let
$\widehat B_i$ the event that $\pi^{x_{i,\delta}}$ stays within
distance $\tilde u$ of $x_{i,\delta}$ at simultaneous regeneration
times up to time $2\tilde t$ and denote by $\widehat C_i(x,m)$ the
event that $\pi^{(x,m)}$ does not coincide with $\pi^{x_{i,\delta}}$
at simultaneous regeneration times before time $2\tilde t$.  In
analogy to the previous notation let $\hat\tau_i^{(x,m)}$ be the first
time that a simultaneous regeneration event occurs after the the
random walk path $\pi^{(x,m)}$ exceeds $(5\cdot i)\tilde u$. Only
considering the random walks at simultaneous regeneration times, we
can for every $\e>0$ estimate a single summand of the sum above by
\begin{align*}
  &\mu_1\l(\bigcap_{i=1}^4B_i \cap \bigcap_{i=1}^4C_i(x,m) \cap \{\tau_4^{(x,m)}<2\tilde t\}\r)\\
  &\leq \mu_1\l(\bigcap_{i=1}^4\widehat B_i \cap\bigcap_{i=1}^4 \widehat C_i(x,m) \cap \{\hat \tau_4^{(x,m)}<(2+\e)\tilde t\}\r) \\
  & \hspace{2em} + 
    \Pr\l( \text{no simultaneous regeneration between time $2\tilde t$ and time $(2+\e)\tilde t$}\r) \\
  &\leq \mu_1\l(\bigcap_{i=1}^4\widehat B_i \cap \bigcap_{i=1}^4 \widehat C_i(x,m) \cap \{\hat \tau_4^{(x,m)}<(2+\e)\tilde t\} \cap\{T^\mathrm{sim}_{\theta_4}-T^\mathrm{sim}_{\theta_4-1}<C\log(\tfrac{1}{\delta})\}\r) \\
  & \hspace{2em} + \delta^4+ 2 \tilde{t} C e^{-c \e \tilde{t}} 
\end{align*}
by using exponential tail bounds for increments of
$T^\mathrm{sim}_\ell - T^\mathrm{sim}_{\ell-1}$, see
Remark~\ref{rem:Tsim.tailbound.m>2}. Here we use that $\theta_4$ is a
stopping time for the joint regeneration construction of the five
walks and that we can choose $C$ so large that
\begin{align*} 
  \Pr\l( T^\mathrm{sim}_{\theta_4}-T^\mathrm{sim}_{\theta_4-1} \geq C\log(\tfrac{1}{\delta}) \r) \leq \delta^4.
\end{align*}
Furthermore, the probability that no simultaneous regeneration occurs
between time $2\tilde t$ and time $(2+\e)\tilde t$ is bounded from
above by
\begin{align*} 
  \sum_{\ell=1}^{\lceil 2\tilde{t} \rceil} \Pr\l( T^\mathrm{sim}_\ell - T^\mathrm{sim}_{\ell-1} > \e \tilde t \r) 
  \leq 2 \tilde{t} C e^{-c \e \tilde{t}} = 2t \delta^{-2} e^{-c \e t/\delta^2} = O(t \delta^4) . 
\end{align*}

Now by the regeneration structure, the only information we gained
about the ``future'' after time $T^{sim}_{\theta_4}$ of the cluster is
that each of the five random walks is at a space-time-point that is
connected to infinity. Therefore, without changing the joint
distribution, the future of the cluster can be replaced by some
identical copy in which all the points the random walks sit in are
connected to infinity. By a coupling argument as in the proof of
Lemma~3.4 in \cite{BCDG13}, the cluster to the right of the middle
line of the third red bar (at horizontal coordinate
$x = 15.5 \tilde{u}$, see Figure~\ref{fig:badcrossing}) can be
replaced by an independent copy and the resulting law on
configurations strictly to the right of this third red bar (i.e.,
$x > 16\tilde{u}$) has total variation distance at most
$2\tilde{t} C e^{-c \tilde{u}}$ to the original law.  Thus
\begin{align*} 
  & \mu_1\l(\bigcap_{i=1}^4\widehat B_i \cap \bigcap_{i=1}^4 \widehat C_i(x,m) \cap \{\hat \tau_4^{(x,m)}<(2+\e)\tilde t\} \cap\{T^\mathrm{sim}_{\theta_4}-T^\mathrm{sim}_{\theta_4-1}<C\log(\tfrac{1}{\delta})\}\r) \\
  & \leq \mu_1\l(\bigcap_{i=1}^3\widehat B_i \cap \bigcap_{i=1}^3 \widehat C_i(x,m) \cap \{\hat \tau_3^{(x,m)}<(2+\e)\tilde t\}\r) \\
  & \hspace{7em} \times \sup_{|y| \le C \log(1/\delta)} \Pr^\mathrm{joint}_y\l( \text{$\widehat{D}$ hits $\tilde{u}$ before $0$} \r) 
                                                                                                                                           \: + \: 2\tilde{t} C e^{-c \tilde{u}}. 
\end{align*}
We use here that the difference between $\pi^{(x,m)}$ and $\pi^{x_{4,\delta}}$, running on an independent 
copy of the percolation cluster (and observed along its regeneration times), behaves like 
the Markov chain $\widehat{D}$ 
from Section~\ref{sect:preliminaries} and the proof of Lemma~\ref{main1}.
Remark~\ref{rem:Dhathittingbounds} gives in particular 
\begin{align*} 
  \sup_{|y| \le C \log(1/\delta)} \Pr^\mathrm{joint}_y\l( \text{$\widehat{D}$ hits $\tilde{u}$ before $0$} \r) 
  \le c \frac{C\log(1/\delta)}{\tilde u} 
  = C' \delta \log(1/\delta)
\end{align*}
with $C'=C'(u)<\infty$.
\smallskip

Combining the above and iterating we get 
\[
\mu_1\l(\bigcap_{i=1}^4B_i \cap \bigcap_{i=1}^4C_i(x,m) \cap 
\{ \tau_4^{(x,m)}<2\tilde t \} \r)\leq \l( 2 C' \delta \log(1/\delta) \r)^4 .
\]
Using this, the term $(\ast\ast)$ is bounded above by
\begin{align*}
  &\mu_1\l(\bigcap_{i=1}^4B_i,\exists (x,m)\in R(\tilde u,\tilde t)\text{ s.t.}\bigcap_{i=1}^4 C_i(x,m)\text{ and }\tau_4^{(x,m)}<2\tilde t\r)\\
  &\leq \sum_{x\in[-\tilde u,\tilde u]\cap \Z}\sum_{m\in[0,\tilde t]\cap \Z} \l( 2 C' \delta \log(1/\delta) \r)^4
         \leq \l( 2 C' \delta \log(1/\delta) \r)^4 \cdot 2\tilde u\tilde t \leq C(u)t\delta \l( \log(1/\delta) \r)^4
\end{align*}
This implies that condition \eqref{eq:condT1+} 
is satisfied.

\subsubsection{Checking condition \texorpdfstring{$(B_1')$}{B1}}
We fix $t>\beta>0$ and $t_0,a\in \R$. We want to show that for each
$\e'>0$ there exists $\e>0$ independent of $t,t_0$ and $a$, such that
\[\mu_\delta(\eta(t_0,t;a-\e,a+\e)>1)=\mu_1(\eta(\tilde t_0,\tilde t;\tilde a-\tilde\e,\tilde a+\tilde\e)>1)<\e',\]
for all $\delta>0$ sufficiently small.  First we assume that $\tilde
t_0=n_0\in \Z$. In this case only paths that start from the interval
$[\tilde a-\tilde\e,\tilde a+\tilde\e]\cap \Z$ at time $n_0$ are
counted by $\eta$.  Therefore
\begin{align*}
  &\mu_1(\eta(n_0,\tilde t;\tilde a-\tilde\e,\tilde a+\tilde\e)>1)\\
  &\leq\sum_{\{x,x+1\}\subset[\tilde a-\tilde\e,\tilde a+\tilde\e]\cap \Z}\Pr\l(\pi^{(x,n_0)}(k)\neq\pi^{(x+1,n_0)}(k) \text{ for all }k\in [n_0,n_0+\lfloor\tilde t\rfloor]\r).
\end{align*}
By Lemma~\ref{main1} we get that
\[\Pr\l(\pi^{(x,n_0)}(k)\neq\pi^{(x+1,n_0)}(k) \text{ for all }k\in [n_0,n_0+\lfloor \tilde t\rfloor]\r)\leq \frac{C}{\sqrt{\tilde t}}\]
for some large constant $C$ and
\begin{equation*}
  \mu_1(\eta(n_0,\tilde t;\tilde a-\tilde\e,\tilde a+\tilde\e)>1)\leq \frac{2\tilde \e C}{\sqrt{\tilde t}}\leq \frac{2 v \e C}{\sqrt{t}}\leq \frac{2 v \e C}{\sqrt{\beta}},
\end{equation*}
which is smaller than $\e'$ if $\e<\frac{\e' \sqrt{\beta}}{2 v C}$.\\
\text{ }\\
If $\tilde{t}_0\in(n_0,n_0+1)$ for some $n_0\in\N$, it is enough to show that $\mu_1(\eta(\tilde t_0,\tilde t;\tilde a-2\tilde\e,\tilde a+2\tilde\e)>1)<\e'$, which is true by similar estimates as above.

\subsubsection{Checking condition \texorpdfstring{$(E_1')$}{E1'}}
In order to verify condition $(E_1')$ we need to prove a statement
similar to Lemma\ 6.2 in \cite{NewmanRavishankarSun2005} which is
formulated in Lemma \ref{locfin} below. This can be done by adapting
Lemma\ 2.7 in \cite{NewmanRavishankarSun2005} to our case (see Lemma
\ref{contains0} below). The rest of the proof follows by more general
results, proved in \cite[Section~6]{NewmanRavishankarSun2005}) and
does not need adaptation.
\begin{lemma}
  \label{contains0}
  Recall the collection of paths $\mathbf{\Gamma}$ from
  \eqref{def:Gamma}. For $A\subset \Z$ and $m,n\in \N,\,m>n$, we define
  \[\mathbf{\Gamma}^{A,n}_m:=\{\pi^{(x,n)}(m):x\in A, (x,n)\in \CC\}.\]
  If $n=0$ we simply write $\mathbf{\Gamma}^{A}_m:=\mathbf{\Gamma}^{A,0}_m$. Then 
  \[p_m:=\Pr\l(0\in\mathbf{\Gamma}^{\Z}_m\r)\leq \frac{C}{\sqrt{m}},\]
  for some constant $C$ independent of time.
\end{lemma}
\begin{proof} 
  Pick $M \in \N$. 
  Let 
  $B_M := \{0,1,\dots,M-1\}$ and in order to simplify notation define
  \begin{align*}
    \mathbf{\Gamma}^{A}_m(x):=
    \begin{cases}
      1, &\text{ if } x \in \mathbf{\Gamma}^A_m\\
      0, &\text{otherwise}
    \end{cases}
  \end{align*}
  for $A\subset\Z$. Using the translation invariance of $\Pr$ we
  obtain
  \[
  e_m(B_M):=\E[|\mathbf{\Gamma}^\Z_m\cap B_M|]
  =\E\l[\sum_{x\in B_M}\mathbf{\Gamma}^\Z_m(x)\r]
  =\sum_{x\in B_M}\E\l[\mathbf{\Gamma}^\Z_m(x)\r]=p_m \cdot M.\] 
  Furthermore, 
  \[
  e_m(B_M)\leq\sum_{k\in\Z}\E[|\mathbf{\Gamma}^{B_M+kM}_m\cap B_M|]
  =\sum_{k\in\Z}\E[|\mathbf{\Gamma}^{B_M}_m\cap (B_M-kM)|]=\E[|\mathbf{\Gamma}^{B_M}_m|].
  \]
  Now the difference $M-|\mathbf{\Gamma}^{B_M}_m|$ is larger than the number
  of nearest neighbour pairs that coalesced before time $m$. Using
  the translation invariance of $\Pr$ again we get that
  \begin{align*}
    \E[M-|\mathbf{\Gamma}^{B_M}_m|]&\geq\sum_{x=0}^{M-2}\E[\ind{\pi^{(x,0)}(t)=\pi^{(x+1,0)}(t)\text{ for some }t \in \{1,2,\dots,m\}}]\\
                                   &=(M-1)\Pr[\pi^{(0,0)}(t)=\pi^{(1,0)}(t)\text{ for some }t \in \{1,2,\dots,m\}]
  \end{align*}
  Lemma~\ref{main1} gives 
  \begin{align*}
    \E[|\mathbf{\Gamma}^{B_M}_m|]&\leq M-(M-1)\Pr[\pi^{(0,0)}(t)=\pi^{(1,0)}(t)\text{ for some }t \in \{1,2,\dots,m\}]\\
                                 &\leq M-(M-1)\l(1-\frac{C}{\sqrt{m}}\r) < 1 + M\frac{C}{\sqrt{m}}
  \end{align*}
  and therefore
  \[p_m<\frac{1}{M}+\frac{C}{\sqrt{m}}.\]
  This yields the claim since $M$ can be chosen arbitrarily large.
\end{proof}
\medskip

Now we are ready to prove our analogon of \cite[Lemma~6.2]{NewmanRavishankarSun2005}. 
Recall the notation $\CX^{t_0^-}$ from \eqref{Xs-definition}. 
\begin{lemma}
  \label{locfin}
  Let $\CZ_{t_0}$ be a subsequential limit of $\CX^{t_0^-}_\delta$,
  where $\CX_\delta:=S_{v,\delta}\mathbf{\Gamma}$ and let $\e>0$. The
  intersection of the paths in $\CZ_{t_0}$ with the line
  $\R \times \{t_0+\e\}$ is almost surely locally finite.
\end{lemma}
\begin{proof} Let $\CZ_{t_0}$ be the weak limit of a sequence
  $(\CX^{t_0^-}_{\delta_n})_n$ and let $\CZ_{t_0}(t_0+\e)$ be the intersection of 
  all paths in $\CZ_{t_0}$ with the line $t_0+\e$; define $\CX^{t_0^-}_{\delta_n}(t_0+\e)$ 
  analogously. Then $\CZ_{t_0}$ and $\CX^{t_0^-}_{\delta_n}(t_0+\e)$, $n \in \N$ are 
  random variables with values in $(\CP,\rho_\CP)$, where $\CP$ is the space of all 
  compact subsets of $(R^2_c,\rho)$, metrized with the induced Hausdorff metric $\rho_\CP$. 

  Since for all $a,b\in\R,\;a<b$ the set $\{K\in
  (\CP,\rho_\CP):|K\cap(a,b)\times \R|\geq k\}$ is an open set in
  $(\CP,\rho_\CP)$, we get that
  \begin{align*}
    \E[|\CZ_{t_0}(t_0+\e)\cap(a,b)\times\R|]&=\sum_{k=1}^{\infty}\Pr[|\CZ_{t_0}(t_0+\e)\cap(a,b)\times\R|\geq k]\\
                                            &\leq\sum_{k=1}^{\infty}\liminf_{n\lra \infty}\Pr[|\CX^{t_0^-}_{\delta_n}(t_0+\e)\cap(a,b)\times\R|\geq k]\\
                                            &\leq\liminf_{n\lra \infty}\E[|\CX^{t_0^-}_{\delta_n}(t_0+\e)\cap(a,b)\times\R|]
                                                   \leq\frac{C(b-a)}{\sqrt{\e}},
  \end{align*}
  where we used the Portmanteau theorem in the second line.  The last
  inequality holds true by Lemma \ref{contains0}, since (recall the
  scaling notation $\tilde x$, etc.\ introduced before
  \eqref{eq:condT1+})
  \begin{align*}
    &\E[|\CX^{t_0^-}_{\delta}(t_0+\e)\cap(a,b)\times\R|]\\
    &\leq\E\l[\sum_{x\in (\tilde a,\tilde b)\cap \Z}\mathbf{\Gamma}^\Z_{\tilde{t}_0+\tilde\e}(x)\r] \leq \sum_{x\in (\tilde a,\tilde b)\cap \Z} \E[\mathbf{\Gamma}^\Z_{\tilde\e}(x)]\leq\frac{C(\tilde b-\tilde a)}{\sqrt{\tilde \e}}\leq\frac{C(b-a)}{\sqrt{\e}}.
  \end{align*}
  Strictly speaking, since $\tilde\e = \e \delta^{-2}$ need not be an
  integer time, we should estimate
  $|\CX^{t_0^-}_{\delta}(t_0+\e)\cap(a,b)\times\R| \leq \sum_{x\in
    (\tilde a,\tilde b)\cap \Z} \l( \mathbf{\Gamma}^\Z_{\lfloor \tilde\e
    \rfloor}(x) + \mathbf{\Gamma}^\Z_{\lceil \tilde\e \rceil}(x)\r)$
  but this changes only the constant.
\end{proof}
\medskip

Using Lemma~\ref{locfin}, Condition $(E_1')$ can then be proved using the 
strategy from \cite{NewmanRavishankarSun2005}, see Lemma~6.3 there.  

\section{Outlook}
\label{sect:outlook}
Our result can be seen as a convergence result for the space-time
embeddings of ``all ancestral lines'' in a discrete time contact
process.  More precisely, define the contact process as follows:
$(\eta_n^A)_{n \ge m}$ starting at time $m \in \Z$ from the set $A$ as
\begin{align*}
  \eta_m^A (y) & =\indset{A}(y), \; y \in \Z^d, \\
  \intertext{and for $n \ge m$}
  \eta_{n+1}^A(x) & =
                    \begin{cases}
                      1 & \text{if $\omega(x,n+1)=1$ and $\eta_n^A(y)=1$ for some $y \in \Z^d$
                        with $\norm{x-y} \le 1$}, \\
                      0 & \text{otherwise},
                    \end{cases}
\end{align*}
i.e., $\eta_n^A(y)=1$ if and only if there is an open path from
$(x,m)$ to $(y,n)$ for some $x\in A$.
\smallskip

By monotonicity, $\CL(\eta_n^{\Z^d}) \to \nu$ as $n\to\infty$, where
the convergence is weak convergence and $\nu \in \CM(\{0,1\}^{\Z^d})$
is the upper invariant measure, cf \cite{Liggett:1999}.  \smallskip
Note that the percolation cluster is given as the time-reversal of the
stationary process $\eta$. More precisely process
$\xi\coloneqq (\xi_n)_{n \in \Z}$ defined by $\xi_n(x)=\eta_{-n}(x)$,
i.e.\ $\xi_n(x)=1$ iff $-\infty \to^\omega (x,-n)$ (defined as
$\bigcap_{m \ge n} \big\{ \Z^d \times \{-m\} \to^\omega (x,-n)
\big\}$)
describes the percolation cluster in the sense that $\xi_n(x) = 1$ if
and only if $x \in \CC$. See \cite{BCDG13} for more details.
\smallskip Hence, the coalescing walkers on the backbone of the
cluster correspond to space-time embeddings of all ancestral
lines. One may then apply our convergence result to investigate the
behaviour of interfaces in the discrete time contact process
analogously to \cite[Theorem~7.6 and
Remark~7.7]{NewmanRavishankarSun2005}.
For the continuous-time contact process, interfaces and their scaling limits
were analyzed in \cite{MountfordValesin:2016, Valesin:2010} (without explicitly using a Brownian web limit). 
\medskip

As noted in Remark~\ref{rem:aftermain}, Theorem~\ref{main2} is an
``annealed'' limit theorem and it would be interesting to prove an
analogous ``quenced'' result. Since Lemma~\ref{main1} is a key
ingredient in the proof, we this would require a quenched analogue of
\eqref{eq:tailbounds}. In this direction, we conjecture
(based on simulations) that in $d=1$, 
\begin{align*} 
  \lim_{n\to\infty} \sqrt{n} \, P_\omega(T^{(z_1, z_2)}_{\textit{meet}} > n) 
\end{align*}
exists for $\Pr$-a.a.\ $\omega$ (and is a non-trivial function of $\omega$).
\bigskip

\paragraph{\bf Acknowledgements} 
The authors would like to thank Rongfeng Sun for his many helpful comments 
on the manuscript. We also thank an anonymous referee whose suggestions 
helped to improve the presentation. 
M.B.\ and S.S.\ were supported by DFG priority programme SPP~1590  
through grants BI\ 1058/3-1 and BI\ 1058/3-2, 
N.G.\ through grant GA\ 582/7-2. 


\bibliographystyle{plainnat}
\bibliography{bgs-revision-arxiv1812.03733v2.bib}

\begin{appendix}
  \section{Proof that $f$ from \eqref{eq:superharmcandidate} is 
    superharmonic for $\widehat{\Psi}^\mathrm{joint}_\mathrm{diff}$}

  \label{sect:proof.superharmcandidate}

  Consider $x>0$, say.  
  \begin{align}
    \sum_{y \in \Z} \widehat{\Psi}^\mathrm{joint}_\mathrm{diff}(x,y) \big( f(y) - f(x) \big) 
    & \le \sum_{y : |y-x| \ge x/3} \widehat{\Psi}^\mathrm{joint}_\mathrm{diff}(x,y) \big( c|y| + c|x| \big) \notag \\
    \label{eq:fsuperharm1}
    & \hspace{2em} + \sum_{y : |y-x| < x/3} \widehat{\Psi}^\mathrm{joint}_\mathrm{diff}(x,y) \big( f(y) - f(x) \big) 
  \end{align}
  The first term on the right-hand side is bounded by $C_2 e^{-c_2 x}$ for suitable $c_2, C_2 \in (0,\infty)$ 
  by \eqref{eq:Psijointtail}, for the second term 
  we use Taylor expansion to write it with some $\xi_{x,y} \in (x\wedge y, x\vee y)$ as 
  \begin{align}
    \sum_{y : 0 < |y-x| < x/3} \widehat{\Psi}^\mathrm{joint}_\mathrm{diff}(x,y) \Big( (y-x) f'(x) + \frac12(y-x)^2 
    f''(\xi_{x,y}) \Big)
  \end{align}
  Since $|\xi_{x,y} - x|<x/3$ we have $f''(\xi_{x,y}) \le f''(4x/3) = -2 e^{-4 c_1 x/3} \exp\big( 2 e^{-4 c_1 x/3}/c_1\big)$
  and thus 
  \begin{align}
    &\sum_{y : 0 < |y-x| < x/3} \widehat{\Psi}^\mathrm{joint}_\mathrm{diff}(x,y) \frac12(y-x)^2 f''(\xi_{x,y}) 
      \notag \\
    & \hspace{2em} \leq -2 e^{-4 c_1 x/3} \exp\big( 2 e^{-4 c_1 x/3}/c_1\big) 
      \sum_{y : 0 < |y-x| < x/3} \widehat{\Psi}^\mathrm{joint}_\mathrm{diff}(x,y) \frac12(y-x)^2 
      \notag \\
    & \hspace{2em} \leq 
      - \frac{\tilde{\sigma}^2}2 e^{-4 c_1 x/3} \exp\big( 2 e^{-4 c_1 x/3}/c_1\big) 
  \end{align}
  (recall \eqref{eq:Dhatcondmeanandvar}).

  Furthermore, by Lemma~\ref{lemma:BCDG.Lemma3.4} and \eqref{eq:Psijointtail}, 
  \begin{align}
    & \sum_{y : 0 < |y-x| < x/3} \widehat{\Psi}^\mathrm{joint}_\mathrm{diff}(x,y) (y-x)
      \notag \\
    & \hspace{3em} \le C_3 e^{-c_3 x} + \sum_{y : 0 < |y-x| \ge x/3} \widehat{\Psi}^\mathrm{joint}_\mathrm{diff}(x,y) |y-x| 
      \le C_4 e^{-c_4 x}
  \end{align}
  for suitable $c_3, C_3, c_4, C_4 \in (0,\infty)$,

  Combining, we see that the right-hand side of \eqref{eq:fsuperharm1} is negative if we 
  choose $x_0>0$ so large and $c_1>0$ so small that (note $f'(x) = \exp\big( 2 e^{-c_1 x}/c_1\big)$)
  \begin{align} 
    C_2 e^{-c_2 x} + \exp\big( 2 e^{-c_1 x}/c_1\big) C_4 e^{-c_4 x} 
    - \frac{\tilde{\sigma}^2}2 e^{-4 c_1 x/3} \exp\big( 2 e^{-4 c_1 x/3}/c_1\big) < 0 
  \end{align}
  holds for all $x \ge x_0$.
\end{appendix}

\end{document}